\documentclass[oneside,english,reqno]{amsart}
\usepackage{hyperref}

\usepackage{amsmath}
\usepackage{amsfonts,amssymb,amsmath,latexsym,wasysym,mathrsfs,bbm,stmaryrd}
\usepackage[all]{xy}

\usepackage[mathcal]{euscript}
\usepackage{times}

\def\1{\mathbbm{1}}

\long\def\symbolfootnote[#1]#2{\begingroup%
\def\thefootnote{\fnsymbol{footnote}}\footnote[#1]{#2}\endgroup}

\usepackage{amsthm}

\newtheorem{theorem}{Theorem}[section]

\newtheorem{proposition}[theorem]{Proposition}
\newtheorem{definition}[theorem]{Definition}

\newtheorem{corollary}[theorem]{Corollary}

\newtheorem{lemma}[theorem]{Lemma}
\newtheorem{notation}[theorem]{Notation}
\numberwithin{equation}{section} 

\theoremstyle{remark}
\newtheorem{example}[theorem]{Example}
\newtheorem{remark}[theorem]{Remark}

\numberwithin{equation}{section}

\begin{document}

\title{The Complex-Time Segal--Bargmann Transform}

\title{The Complex Time Segal--Bargmann Transform}
\author{Bruce K. Driver}
\address{Department of Mathematics \\
University of California, San Diego \\
La Jolla, CA 92093-0112 \\
\texttt{bdriver@ucsd.edu}}
\author{Brian C. Hall}
\thanks{Supported in part by NSF Grant DMS-1301534}
\address{Department of Mathematics \\
University of Notre Dame \\
Notre Dame, IN 46556 \\
\texttt{bhall@nd.edu}}
\author{Todd Kemp}
\thanks{Supported in part by NSF CAREER Award DMS-1254807 and NSF Grant DMS-1800733}
\address{Department of Mathematics \\
University of California, San Diego \\
La Jolla, CA 92093-0112 \\
\texttt{tkemp@ucsd.edu}}

\begin{abstract}

We introduce a new form of the Segal--Bargmann transform for a Lie group $K$
of compact type. We show that the heat kernel $(\rho_{t}(x))_{t>0,x\in K}$ has
a space-time analytic continuation to a holomorphic function
\[
(\rho_{\mathbb{C}}(\tau,z))_{\mathrm{Re}\,\tau>0,z\in K_{\mathbb{C}}}
\]
where $K_{\mathbb{C}}$ is the complexification of $K$. The new transform is
defined by the integral
\[
(B_{\tau}f)(z)=\int_{K}\rho_{\mathbb{C}}(\tau,zk^{-1})f(k)\,dk,\quad z\in
K_{\mathbb{C}}.
\]
If $s>0$ and $\tau\in\mathbb{D}(s,s)$ (the disk of radius $s$ centered at
$s$), this integral defines a holomorphic function on $K_{\mathbb{C}}$ for
each $f\in L^{2}(K,\rho_{s})$. We construct a heat kernel density $\mu_{s,\tau}$
on $K_{\mathbb{C}}$ such that, for all $s,\tau$ as above,
$B_{s,\tau}:=B_{\tau}|_{L^{2}(K,\rho_{s})}$ is an isometric isomorphism from
$L^{2}(K,\rho_{s})$ onto the space of holomorphic functions in
$L^{2}(K_{\mathbb{C}},\mu_{s,\tau})$. When $\tau=t=s$, the transform $B_{t,t}$
coincides with the one introduced by the second author for compact groups and
extended by the first author to groups of compact type. When $\tau=t\in
(0,2s)$, the transform $B_{s,t}$ coincides with the one introduced by the
first two authors.

\end{abstract}

\maketitle

\tableofcontents

\section{Introduction}

\subsection{The Classical Segal--Bargmann Transform\label{SBclassical.sec}}

This paper concerns a generalization of the Segal--Bargmann transform over
compact-type Lie groups, to allow the time parameter of the transform to be
complex. We begin by briefly discussing the history of the transform. For
$t>0$ and $d\in\mathbb{N}$, let $\rho_{t}$ denote the variance-$t$ Gaussian
density on $\mathbb{R}^{d}$:
\[
\rho_{t}(x)=(2\pi t)^{-d/2}\exp\left(  -\frac{|x|^{2}}{2t}\right)  .
\]
This is the \emph{heat kernel} on $\mathbb{R}^{d}$: the solution $u$ of the
heat equation $\partial_{t}u=\frac{1}{2}\Delta u$ with (sufficiently
integrable) initial condition $f$ is given in terms of $\rho_{t}$ by
\begin{equation}
u(t,x)=(\rho_{t}\ast f)(x)=\int_{\mathbb{R}^{d}}\rho_{t}(x-y)f(y)\,dy.
\label{e.heat.eq}%
\end{equation}

The function $\rho_{t}$ admits an explicit entire analytic continuation to
$\mathbb{C}^{d}$, which we call $(\rho_{t})_{\mathbb{C}}$: it is simply the
function
\[
(\rho_{t})_{\mathbb{C}}(z)=(2\pi t)^{-d/2}\exp\left(  -\frac{z\cdot z}%
{2t}\right)  ,
\]
where $z\cdot z=\sum_{j=1}^{d}z_{j}^{2}.$ If $f\in L_{\mathrm{loc}}%
^{1}(\mathbb{R}^{d})$ and of sufficiently slow growth, then the integral
\begin{equation}
(B_{t}f)(z):=\int_{\mathbb{R}^{d}}(\rho_{t})_{\mathbb{C}}(z-y)f(y)\,dy
\label{e.SB.original}%
\end{equation}
converges and defines an entire holomorphic function on $\mathbb{C}^{d}$.

The map $f\mapsto B_{t}f$ is equivalent to the \textbf{Segal--Bargmann
transform}, invented and explored by the eponymous authors of
\cite{Bargmann1961,Bargmann1962,Segal1962,Segal1963,Segal1978}. Note that
neither Segal nor Bargmann explicitly connected the transform to the heat
kernel, nor did they write the transform precisely as in (\ref{e.SB.original}%
). Nevertheless, their transforms can easily be rewritten in the form
(\ref{e.SB.original}) by simple changes of variable; cf.\ \cite{Hall2000}.

We consider also the heat kernel on $\mathbb{C}^{d}\cong\mathbb{R}^{2d}$ (with
time-parameter rescaled by a factor of 2), which we refer to as $\mu_{t}$:
\[
\mu_{t}(z)=(\pi t)^{-d}\exp(-|z|^{2}/t).
\]
(Note that the real, positive function $\mu_{t}$ on $\mathbb{C}^{d}$ is not
the same as the holomorphic function $(\rho_{t})_{\mathbb{C}}$.) The main
theorem is that $B_{t}$ is an isometric isomorphism from $L^{2}(\mathbb{R}%
^{d},\rho_{t})$ onto $\mathcal{H}L^{2}(\mathbb{C}^{d},\mu_{t})$ --- the space
of holomorphic functions in $L^{2}(\mathbb{C}^{d},\mu_{t})$. For more
information about the classical Segal--Bargmann transform, see, for example,
\cite{Hall2000,HallQM}.

\subsection{The Segal--Bargmann Transform for Lie Groups of Compact Type}

In \cite{Hall1994}, the second author introduced an analog of the
Segal--Bargmann transform on an arbitrary compact Lie group. Then, in
\cite{Driver1995}, the first author extended the results of \cite{Hall1994} to
a Lie group $K$ of compact type (Section \ref{sec.ctlgtc}), a class that
includes both compact groups and $\mathbb{R}^{d}$. The idea of \cite{Hall1994}
and \cite{Driver1995} is the same as in the $\mathbb{R}^{d}$ case: the heat
kernel $\rho_{t}$ on $K$ has an entire analytic continuation $(\rho
_{t})_{\mathbb{C}}$ to the \emph{complexification} $K_{\mathbb{C}}$ of $K$.
The transform $B_{t}$ is defined by the group convolution formula generalizing
\eqref{e.SB.original}:
\begin{equation}
(B_{t}f)(z)=\int_{K}\,(\rho_{t})_{\mathbb{C}}(zk^{-1})f(k)\,dk.
\label{e.SB.Hall}%
\end{equation}
The theorem is that $B_{t}$ is an isometric isomorphism from $L^{2}(K,\rho
_{t})$ onto the holomorphic space $\mathcal{H}L^{2}(K_{\mathbb{C}},\mu_{t})$,
where $\mu_{t}$ is the (time-rescaled) heat kernel on $K_{\mathbb{C}}$. If
$K=\mathbb{R}^{d}$, then $B_{t}$ is precisely the classical Segal--Bargmann
transform of Section \ref{SBclassical.sec}.

Later, in \cite{Driver1999,Hall2001b}, the authors made a further
generalization related to the time parameter $t$. One can use a different time
$s\neq t$ to measure the functions $f$ in the domain, while still using the
analytically continued heat kernel at time $t$ to define the transform, as in
\eqref{e.SB.Hall}. The resulting map,
\[
B_{s,t}\colon L^{2}(K,\rho_{s})\rightarrow\mathcal{H}L^{2}(K_{\mathbb{C}}%
,\mu_{s,t})
\]
is still an isometric isomorphism for an appropriate two-parameter heat kernel
density $\mu_{s,t}$, provided $0<t<2s$. Note that the formula for the
transform $B_{s,t}$ does not depend on $s$; this parameter only indicates the
inner product to be used on the domain and range spaces. In the special case
that $K=\mathbb{R}^{d}$, the two-parameter heat kernel density $\mu_{s,t}$ in
the range is a Gaussian measure with different variances in the real and
imaginary directions. (Take $u=0$ in (\ref{muRk}) below.)

\begin{remark}
\label{pointwise.rem} For a complex manifold $M$, let $\mathcal{H}(M)$ denote
the space of holomorphic functions on $M$. If $\mu$ is a measure on $M$ having
a strictly positive, continuous density with respect to the Lebesgue measure
in each holomorphic local coordinate system, it is not hard to show that
$\mathcal{H}L^{2}(M,\mu):=\mathcal{H}(M) \cap L^{2}(M,\mu)$ is a closed
subspace of $L^{2}(M,\mu)$ and is therefore a Hilbert space. Furthermore, the
\emph{pointwise evaluation map} $F\mapsto F(z)$ is continuous for each $z\in
M$, and the norm of this functional is locally bounded as a function of $z$.
(See, for example, Theorem 3.2 and Corollary 3.3 in \cite{Driver2015} or
Theorem 2.2 in \cite{Hall2000}.)
\end{remark}

\subsection{The Complex-Time Segal--Bargmann Transform}

The topic of the present paper is a new generalization that modifies the
transform $B_{s,t}$ as well; in particular, we show that the time parameter
$t$ can also be extended into the complex plane, and there is still an
isomorphism between real and holomorphic $L^{2}$ spaces of associated heat
kernel measures. This generalization is natural and, in a certain sense, a
completion of Segal--Bargmann transform theory, as explained below. (See
Theorem \ref{t.Ad(K)-inv.inn.prod}.)

Let $K$ be a compact-type Lie group with Lie algebra $\mathfrak{k}$, and fix
an $\mathrm{Ad}(K)$-invariant inner product $\left\langle \cdot,\cdot
\right\rangle _{\mathfrak{k}}$ on $\mathfrak{k}$ (Section \ref{sec.ctlgtc}).
This induces a bi-invariant Riemannian metric on $K$, and an associated
Laplace operator $\Delta_{K}$, which is bi-invariant, elliptic, and
essentially self-adjoint in $L^{2}(K)$. There is an associated \textbf{heat
kernel}, $\rho_{t}\in C^{\infty}(K,(0,\infty))$, satisfying%
\begin{equation}
\left(  e^{\Delta_{K}/2}f\right)  (x)=\int_{K}\rho_{t}(xy^{-1})f(y)\,dk\qquad
\text{for all}\;f\in L^{2}(K)\text{ and }t>0. \label{e.d.Heat.kernel}%
\end{equation}
Our first theorem is that the heat kernel can be complexified in both space
and time.

\begin{theorem}
\label{t.complex.time.heat.kernel} Let $K$ be a connected Lie group of compact
type, with a given $\mathrm{Ad}(K)$-invariant inner product on its Lie algebra
$\mathfrak{k}$, and let $(\rho_{t})_{t>0}$ be the associated heat kernel. Let
$\mathbb{C}_{+}$ denote the right half-plane $\{\tau=t+iu\colon t>0,u\in
\mathbb{R}\}$. There is a unique holomorphic function
\[
\rho_{\mathbb{C}}\colon\mathbb{C}_{+}\times K_{\mathbb{C}}\rightarrow
\mathbb{C}%
\]
such that $\rho_{\mathbb{C}}(t,x)=\rho_{t}(x)$ for all $t>0$ and $x\in
K\subset K_{\mathbb{C}}$.
\end{theorem}

\noindent Theorem \ref{t.complex.time.heat.kernel} is proved in Section
\ref{Section SBT}, as part of Theorem \ref{analyticContHeat.thm}.

Following the pattern described above for the $\mathbb{R}^{d}$ case, we make
the following definition.

\begin{notation}
[Complex-time Segal--Bargmann transform]\label{not.ctsbt} For $\tau
\in\mathbb{C}_{+}$ and $z\in K_{\mathbb{C}}$, define
\begin{equation}
\left(  B_{\tau}f\right)  (z):=\int_{K}\rho_{\mathbb{C}}(\tau,zk^{-1}%
)f(k)\,dk\qquad\text{for}\;z\in K_{\mathbb{C}} \label{e.SB.Ctime}%
\end{equation}
for all measurable functions $f\colon K\rightarrow\mathbb{C}$ satisfying
\begin{equation}
\int_{K}\left\vert \rho_{\mathbb{C}}(\tau,zk^{-1})f(k)\right\vert ~dk<\infty.
\label{e.integrable}%
\end{equation}
Further let $\mathcal{D}(B_{\tau})$ denote the vector space of measurable
functions $f\colon K\rightarrow\mathbb{C}$ such that \eqref{e.integrable}
holds for all $z\in K_{\mathbb{C}}$ and such that $B_{\tau}f\in\mathcal{H}%
(K_{\mathbb{C}})$.
\end{notation}

As defined, $\mathcal{D}(B_{\tau})$ is a linear subspace of the measurable
$\mathbb{C}$-valued functions on $K$, and $B_{\tau}\colon\mathcal{D}(B_{\tau})
\rightarrow\mathcal{H}(K_{\mathbb{C}})$ is a linear map. The main theorem of
this paper (Theorem \ref{t.main}) identifies $L^{2}$-Hilbert subspaces of
$\mathcal{D}(B_{\tau})$ and $\mathcal{H}(K_{\mathbb{C}})$ which are unitarily
equivalent to one another under the action of $B_{\tau}$. To describe the
relevant subspaces of $\mathcal{H}(K_{\mathbb{C}})$ we need a little more notation.

As on $K$, we fix once and for all a right Haar measure $\lambda$ on
$K_{\mathbb{C}}$, and typically write $dz$ for $\lambda(dz)$ and
$L^{2}(K_{\mathbb{C}})$ for $L^{2}(K_{\mathbb{C}},\lambda)$. When $s>0$, let
$\mathbb{D}(s,s)\subset\mathbb{C}_{+}$ denote the open disk of radius $s$
centered at $s$.

\begin{definition}
\label{d.Lap.rsu}Let $s>0$ and $\tau=t+iu\in\mathbb{C}$. The \textbf{$(s,\tau
)$-Laplacian} $\Delta_{s,\tau}$ on $K_{\mathbb{C}}$ is the left-invariant
differential operator
\begin{equation}
\Delta_{s,\tau}=\sum_{j=1}^{\dim\mathfrak{k}}\left[  \left(  s-\frac{t}%
{2}\right)  \tilde{X}_{j}^{2}+\frac{t}{2}\tilde{Y}_{j}^{2}-u\,\tilde{X}%
_{j}\tilde{Y}_{j}\right]  \label{e.def.Delta.s.tau}%
\end{equation}
where $\{X_{j}\}_{j=1}^{\dim\mathfrak{k}}$ is any orthonormal basis of
$\mathfrak{k}$, and $Y_{j}=JX_{j}$ where $J$ is operation of multiplication by
$i$ on $\mathfrak{k}_{\mathbb{C}}=\operatorname*{Lie}(K_{\mathbb{C}})$. Here,
for any $Z\in\mathfrak{k}_{\mathbb{C}},$ we let $\tilde{Z}$ denote the
left-invariant vector field on $K_{\mathbb{C}}$ whose value at the identity is
$X$.
\end{definition}

\begin{remark}
\label{rem.elli} Given $s>0$ and $\tau=t+iu\in\mathbb{C}_{+}$, from
\eqref{e.def.Delta.s.tau}, it is not difficult to show that the operator
$\Delta_{s,\tau}$ is elliptic if and only if
\begin{equation}
\alpha\left(  s,\tau\right)  :=\det\left[
\begin{array}
[c]{cc}%
s-t/2 & -u/2\\
-u/2 & t/2
\end{array}
\right]  =\frac{1}{4}(2st-t^{2}-u^{2})>0. \label{e.alpha}%
\end{equation}
This can be written equivalently as
\begin{equation}
2s>t+u^{2}/t \label{e.s.t.u}%
\end{equation}
or, more succinctly, as $\tau\in\mathbb{D}(s,s)$ (the disk of radius $s$,
centered at $s$). Further notice that $\mathbb{D}(s,s) \uparrow\mathbb{C}_{+}$
as $s\uparrow\infty$.
\end{remark}

If the positivity conditions in Remark \ref{rem.elli} hold, then there exists a heat
kernel density $\mu_{s,\tau}\in C^{\infty}(K_{\mathbb{C}},(0,\infty))$ such
that
\[
\left(  e^{\Delta_{s,\tau}/2}f\right)  (w)=\int_{K_{\mathbb{C}}}\mu_{s,\tau
}(w^{-1}z)\,f(z)\,dz\quad\text{ for all }\;f\in L^{2}(K_{\mathbb{C}}).
\]
We are now prepared to state the main theorem of this paper.

\begin{theorem}
[Complex-time Segal--Bargmann transform]\label{t.main} Let $K$ be a connected,
compact-type Lie group. For $s>0$ and $\tau\in\mathbb{D}(s,s)$, $L^{2}%
(K,\rho_{s}) \subset\mathcal{D}(B_{\tau})$; i.e.,\ $B_{\tau}f$ is holomorphic
on $K_{\mathbb{C}}$ for each $f\in L^{2}(K,\rho_{s})$. The image of $B_{\tau}$
on this domain is $B_{\tau}\left(  L^{2}(K,\rho_{s})\right)  =\mathcal{H}%
L^{2}(K_{\mathbb{C}},\mu_{s,\tau})$. Moreover,
\[
B_{s,\tau}:=B_{\tau}|_{L^{2}(K,\rho_{s})}
\]
is a unitary isomorphism from $L^{2}(K,\rho_{s})$ onto $\mathcal{H}%
L^{2}(K_{\mathbb{C}},\mu_{s,\tau})$.
\end{theorem}

Theorem \ref{t.main} is proved in Section \ref{Section SBT}. The $\tau
=t\in\mathbb{R}$ case of Theorem \ref{largeS.thm} was established in
\cite[Theorem 5.3]{Driver1999}. (See also \cite[Theorem 2.1]{Hall2001b}.)

\begin{remark}
\label{rem.cond}The condition in \cite{Driver1999,Hall2001b} for the
two-parameter Segal--Bargmann transform $B_{s,t}$ to be a well-defined unitary
map was $t>0$ and $s>t/2$, or equivalently $t\in(0,2s)$. It is therefore
natural that, in complexifying $t$ to $\tau$, the optimal condition is that
$\tau\in\mathbb{D}(s,s)$, the most symmetric region whose intersection with
$\mathbb{R}$ is the interval $(0,2s)$.
\end{remark}

In the case that the group $K$ is compact, there is a limiting $s\to\infty$
variant (Theorem \ref{largeS.thm}) of Theorem \ref{t.main}. To state this
variant, as in \cite{Hall1994}, we first introduce a one parameter family of
\textquotedblleft$K$-averaged heat kernels.\textquotedblright

\begin{definition}
\label{d.K-av.heat.kernel}For $t>0$, define the $K$\textbf{-averaged heat
kernel} $\nu_{t}$ on $K_{\mathbb{C}}$ by
\[
\nu_{t}(z)=\int_{K}\mu_{t,t}(zk)\,dk\quad\text{ for all }\;z\in K_{\mathbb{C}}%
\]
where $dk$ denotes the Haar probability measure on $K$.
\end{definition}

\noindent In fact, one can replace $\mu_{t,t}$ by $\mu_{s,\tau}$ for any
$\tau\in\mathbb{D}(s,s)$ in the above integral, and the resulting $K$-averaged
density $\nu_{t}$ is the same: it only depends on $t=\mathrm{Re}\,\tau$; see
Proposition \ref{nutIndep.prop}.

\begin{theorem}
[Large-$s$ limit]\label{largeS.thm} For all $s>0$ and $\tau=t+iu\in
\mathbb{D}(s,s),$ we have $L^{2}(K)=L^{2}(K,\rho_{s})$ and $L^{2}%
(K_{\mathbb{C}},\mu_{s,\tau})=L^{2}(K_{\mathbb{C}},\nu_{t})$ (equalities as
sets). Furthermore, for all $f\in L^{2}(K)$ and all $F\in L^{2}(K_{\mathbb{C}%
},\nu_{t}),$ we have%
\begin{align*}
\lim_{s\rightarrow\infty}\left\Vert f\right\Vert _{L^{2}(K,\rho_{s})}  &
=\left\Vert f\right\Vert _{L^{2}(K)}\\
\lim_{s\rightarrow\infty}\left\Vert F\right\Vert _{L^{2}(K_{\mathbb{C}}%
,\mu_{s,\tau})}  &  =\left\Vert F\right\Vert _{L^{2}(K_{\mathbb{C}},\nu_{t})}.
\end{align*}
It follows that $B_{\infty,\tau}:=\left.  B_{\tau}\right\vert _{L^{2}(K)}$ is
a unitary isomorphism from $L^{2}(K)$ onto $\mathcal{H}L^{2}(K_{\mathbb{C}%
},\nu_{t}).$
\end{theorem}

This theorem is proved in Section \ref{section s-->infinity} below.

\begin{remark}
The unitarity of the map $B_{\infty,\tau}$ was previously established in
\cite[Prop. 2.3]{Florentino2003}. Indeed, this unitarity result follows easily
from the unitarity of the \textquotedblleft$C$-version\textquotedblright%
\ Segal--Bargmann transform in \cite{Hall1994} and the unitarity of the
operator $e^{iu\Delta/2}:L^{2}(K)\rightarrow L^{2}(K).$ The significance of
Theorem \ref{largeS.thm} is that the unitary map $B_{\infty,\tau}$ is, in a
strong sense, the $s\rightarrow\infty$ limit of the unitary map $B_{s,\tau}.$
\end{remark}

\subsection{An Outline of the Proof\label{proofsketch.sec}}

We now give a heuristic proof of the isometricity portion of Theorem
\ref{t.main}, in the Euclidean case $K=\mathbb{R}^{d}$, for motivation. The
argument is a generalization of the method used in the appendix of
\cite{Hall2001a}. By \eqref{e.SB.Ctime}, if we restrict to real time
$\tau=t>0$ and look at the transform $(B_{s,t}f)(x)$ at a point $x\in
\mathbb{R}^{d}$, we simply have $(B_{s,t}f)(x)=\int_{\mathbb{R}^{d}}\rho
_{t}(x-y)f(y)\,dy$; in other words, restricted to real time and $K$,
$B_{s,t}f$ is just the heat operator applied to $f$, $B_{s,t}f=e^{\frac{t}%
{2}\Delta}f$ where $\Delta$ is the standard Laplacian on $\mathbb{R}^{d}$.
Therefore, in general the transform can be described as \textquotedblleft
apply the heat operator, then analytically continue in space and
time\textquotedblright. But if the function $f$ itself already possesses a
holomorphic extension $f_{\mathbb{C}}$ to all of $\mathbb{C}^{d}$ (e.g.,\ if
$f$ is a polynomial), then at least informally we should have
\[
B_{s,\tau}f=e^{\frac{\tau}{2}\Delta}f_{\mathbb{C}},
\]
where now $\Delta$ (the sum of squares of the $\mathbb{R}^{d}$-derivatives) is
acting on functions on $\mathbb{C}^{d}$.

Let $F=B_{s,\tau}f$; we need to compute $|F|^{2}=F\bar{F}$. Since
$f_{\mathbb{C}}$ is holomorphic, we have $\frac{\partial}{\partial x_{j}%
}f_{\mathbb{C}}=\frac{\partial}{\partial z_{j}}f_{\mathbb{C}}$, and so $\Delta
f_{\mathbb{C}}=\sum_{j=1}^{d}\frac{\partial^{2}}{\partial z_{j}^{2}%
}f_{\mathbb{C}}=:\partial^{2}f_{\mathbb{C}}$; similarly $\Delta\bar
{f}_{\mathbb{C}}=\sum_{j=1}^{d}\frac{\partial^{2}}{\partial\bar{z}_{j}^{2}%
}\bar{f}_{\mathbb{C}}=:\bar{\partial}^{2}\bar{f}_{\mathbb{C}}$. Again, since
$f_{\mathbb{C}}$ is holomorphic and $\bar{f}_{\mathbb{C}}$ is antiholomorphic,
$\partial^{2}\bar{f}_{\mathbb{C}}=0=\bar{\partial}^{2}f_{\mathbb{C}}$; so we
have
\begin{equation}
(F\bar{F})=(e^{\frac{\tau}{2}\partial^{2}}f_{\mathbb{C}})(e^{\frac{\bar{\tau}%
}{2}\bar{\partial}^{2}}\bar{f}_{\mathbb{C}})=e^{(\frac{\tau}{2}\partial
^{2}+\frac{\bar{\tau}}{2}\bar{\partial}^{2})}f_{\mathbb{C}}\bar{f}%
_{\mathbb{C}}. \label{e.heuristic.proof.1}%
\end{equation}

Now, we measure $f$ in $L^{2}(\mathbb{R}^{d},\rho_{s})$; setting $x=0$ in the
(additive form of) \eqref{e.d.Heat.kernel} defining the heat operator, we can
compute
\begin{equation}
\Vert f\Vert_{L^{2}(\mathbb{R}^{d},\rho_{s})}^{2}=\int_{\mathbb{R}^{d}}%
\rho_{s}(y)|f(y)|^{2}\,dy=\left(  e^{\frac{s}{2}\Delta}|f|^{2}\right)
(0)=\left(  e^{\frac{s}{2}\Delta}|f_{\mathbb{C}}|^{2}\right)  (0).
\label{e.heuristic.proof.2}%
\end{equation}
Similarly, we measure $F$ in $L^{2}(\mathbb{C}^{d},\mu_{s,\tau})$, meaning
\begin{equation}
\Vert F\Vert_{L^{2}(\mathbb{C}^{d},\mu_{s,\tau})}^{2}=\left(  e^{\frac{1}%
{2}\Delta_{s,\tau}}|F|^{2}\right)  (0). \label{e.heuristic.proof.3}%
\end{equation}
Combining \eqref{e.heuristic.proof.1} and \eqref{e.heuristic.proof.3}, and
commuting partial derivatives to combine the exponentials, we therefore have
\begin{equation}
\Vert B_{s,\tau}f\Vert_{L^{2}(\mathbb{C}^{d},\mu_{s,\tau})}^{2}=\left(
e^{\frac{1}{2}\Delta_{s,\tau}+\frac{\tau}{2}\partial^{2}+\frac{\bar{\tau}}%
{2}\bar{\partial}^{2}}|f_{\mathbb{C}}|^{2}\right)  (0).
\label{e.heuristic.proof.4}%
\end{equation}
Comparing \eqref{e.heuristic.proof.2} with \eqref{e.heuristic.proof.4}, we see
that to prove the isometry in Theorem \ref{t.main}, it suffices to have
\[
s\Delta=\Delta_{s,\tau}+\tau\partial^{2}+\bar{\tau}\bar{\partial}^{2}.
\]

Expressing the operators $\partial^{2}$ and $\bar{\partial}^{2}$ in terms of
real partial derivatives, we can then solve for $\Delta_{s,\tau}$; this is how
\eqref{e.def.Delta.s.tau} arises. In the present Euclidean setting, we have
\begin{equation}
\Delta_{s,\tau}=\sum_{j=1}^{d}\left[  \left(  s-\frac{t}{2}\right)
\frac{\partial^{2}}{\partial x_{j}^{2}}+\frac{t}{2}\frac{\partial^{2}%
}{\partial y_{j}^{2}}-u\frac{\partial^{2}}{\partial x_{j}\partial y_{j}%
}\right]  . \label{e.dstau}%
\end{equation}
As in Remark \ref{rem.elli}, it is easily verified that $\Delta_{s,\tau}$ is
elliptic precisely when $\tau\in\mathbb{D}\left(  s,s\right)  $. Moreover, by
a standard Fourier transform argument, one shows that $e^{\frac{1}{2}%
\bar{\Delta}_{s,\tau}}f=f\ast\mu_{s,\tau}$ where%
\begin{equation}
\mu_{s,\tau}(z)=(2\pi\sqrt{\alpha})^{-d}\left(  -\frac{t/2}{2\alpha}\left\vert
x\right\vert ^{2}-\frac{s-t/2}{2\alpha}\left\vert y\right\vert ^{2}-\frac
{u}{2\alpha}x\cdot y\right)  , \label{muRk}%
\end{equation}
where $z=x+iy\in\mathbb{R}^{d}+i\mathbb{R}^{d}=\mathbb{C}^{d}$, and
$\alpha:=\alpha\left(  s,\tau\right)  $ as in Eq. (\ref{e.alpha}).

When $u=0$, the density $\mu_{s,\tau}$ becomes a product of a Gaussian in the
$x$ variable and a Gaussian in the $y$ variable, but with typically unequal
variances. If $u=0$ and $s=t$, the formula for $\mu_{s,\tau}$ reduces to
\[
\mu_{t,t}(z)=(\pi t)^{-d}e^{-\left\vert z\right\vert ^{2}/t},
\]
which is the density for the standard Segal--Bargmann space over
$\mathbb{C}^{d}$.

For a general Lie group $K$ of compact type, we replace the partial
derivatives in the preceding argument with left-invariant vector fields. The
heuristic argument then goes through unchanged, \textit{except} that we must
remember that left-invariant vector fields do not, in general, commute. Thus,
we must also verify that the particular operators involved in the calculation
do, in fact, commute, allowing us to combine the exponents as above. For this,
we need to use an inner product on the Lie algebra of $K$ that is
Ad-invariant; this is the reason for the assumption that $K$ be of compact type.

Most of this paper is devoted to making the above argument rigorous. The key
is to introduce a dense subspace (consisting of matrix entries; see\ Section
\ref{Section Matrix Entries}) of the domain Hilbert space on which integration
against the heat kernel can be computed rigorously by a power series in the
relevant Laplacian. This argument can be found in Section \ref{Section SBT}.

The operator $\Delta_{s,\tau}$ was the starting point for the current
investigation. It is the Laplacian for a left-invariant Riemannian metric on
$K_{\mathbb{C}}$ for which the corresponding inner product on the Lie algebra
is invariant under the adjoint action of $K$. While the Lie algebra of the
complexified Lie group $K_{\mathbb{C}}$ does not possess a fully $\mathrm{Ad}%
$-invariant inner product (unless $K$ is commutative), it does possess many
inner products that are invariant under the adjoint action of $K$. These are
the most natural from the perspective of diffusion processes, particularly in
high dimension (cf.\ \cite{Kemp2015b}). In fact, there is a natural three
(real) parameter family of $\mathrm{Ad}(K)$-invariant inner products on
$\mathrm{Lie}(K_{\mathbb{C}})$ (see \eqref{e.(a,b,c)->(s,t,u)} for the
relation to the Segal--Bargmann transform parameters $s$ and $\tau=t+iu$). In
the case that $K$ is simple, this is a complete characterization of all such
invariant inner products; this is the statement of Theorem
\ref{t.Ad(K)-inv.inn.prod} below. It was this fact that led the authors
backward to discover the complex-time Segal--Bargmann transform, which is
therefore a natural completion of the versions of the transform previously
introduced by Segal, Bargmann, and the first two authors of the present paper.

\subsection{Motivation\label{section motivation}}

In the case $K=\mathrm{U}(n)$ and $K_{\mathbb{C}}=\mathrm{GL}(n;\mathbb{C}),$
we may give one motivation for the complex-time Segal--Bargmann transform as
follows: choosing matrices at random from $\mathrm{GL}(n;\mathbb{C})$ with
distribution $\mu_{s,\tau}$ is an interesting random matrix model and the
transform is a tool for studying that model. We now elaborate on this
statement, starting by thinking of the heat kernel measure on $\mathrm{GL}%
(n;\mathbb{C})$ as giving a random matrix model. The heat kernel measure
$\mu_{s,\tau}(g)~dg$ on $\mathrm{GL}(n;\mathbb{C})$ is just the group analog
of a Gaussian measure on its Lie algebra, the space of all $n\times n$
matrices. In the two-parameter case (i.e., with $\tau=t\in\mathbb{R}$), the
Gaussian measure is a scaled version of the Ginibre ensemble. In the large-$n$
limit, the eigenvalues of a random matrix chosen according to this Gaussian
measure are uniformly distributed on an ellipse with axes lying along the real
and imaginary axes. One can certainly add a third parameter to the Gaussian
measure, but one does not really get anything new by doing so: The resulting
random matrix is just the two-parameter case multiplied by a fixed complex
number. Thus, the limiting eigenvalue distribution is uniform over an ellipse
in $\mathbb{C}$---but an ellipse that has been rotated so its axes no longer
lie along the real and imaginary axes.

For the heat kernel measure on $\mathrm{GL}(n;\mathbb{C}),$ the problem is
much richer. In the two-parameter case (i.e., with $\tau=t\in\mathbb{R}$), the
second and third authors have used \cite{Hall2018} the large-$n$
Segal--Bargmann transform developed in \cite{Biane1997b,Driver2013,Ho2016} to
identify the domain $\Sigma_{s,t}$ in $\mathbb{C}$ on which the
\textquotedblleft Brown measure\textquotedblright\ of the limiting object is
supported. We expect that this is the domain into which the eigenvalues of
random matrices chosen from $\mathrm{GL}(n;\mathbb{C})$ and distributed as
$\mu_{s,t}$ cluster in the $n\rightarrow\infty$ limit. In the case $s=t,$ the
authors then computed the Brown measure---not just its support---in
\cite{Driver2019}.

Already in the two-parameter case, the domains $\Sigma_{s,t}$ display an
interesting structure, changing from simply connected to doubly connected at
$s=4.$ If we then allow $\tau$ to be complex, the associated random matrix
model is no longer just a complex number times the two-parameter case. Thus
the domain into which the eigenvalues cluster will not be simply a rotation of
$\Sigma_{s,t}$. Rather, simulations indicated that the domain gets twisted
around in a much more complicated (and therefore interesting) way. The
large-$n$ limit of the complex-time Segal--Bargmann transform has already been
developed in \cite{chan2018}. We expect that this limiting transform will be
an important tool in studying the large-$n$ eigenvalue distribution of
$\mu_{s,\tau}$, in the same way that the large-$n$ limit of the two-parameter
transform was used in \cite{Hall2018}.

In the rest of this subsection, we provide motivation for considering the
complex-time transform for a fixed, finite-dimensional Lie group of compact
type. The Segal--Bargmann transform $(B_{\tau}f)(z)$ is computed by
integration of $f$ against the function
\begin{equation}
\chi_{\tau}^{z}(x):=\rho_{\mathbb{C}}(\tau,x^{-1}z). \label{coherentKc}%
\end{equation}
These functions may be thought of as \textquotedblleft coherent
states\textquotedblright\ on $K$. In the case $K=\mathbb{R}^{1},$ coherent
states are often defined as minimum uncertainty states, namely those giving
equality in the classic Heisenberg uncertainty principle. There is, however, a
stronger form of the uncertainty principle, due to Schr\"{o}dinger
\cite{Schrodinger1930}, which says that
\begin{equation}
\left(  \Delta_{\chi}X\right)  ^{2}(\Delta_{\chi}P)^{2}\geq\frac{\hbar^{2}}%
{4}+\left\vert \mathrm{Cov}_{\chi}(X,P)\right\vert ^{2},
\label{schrUncertainty}%
\end{equation}
where $\Delta_{\chi}X$ is the uncertainty of the observable $X$ in state
$\chi$, and
\[
\mathrm{Cov}_{\chi}(X,P):=\left\langle (XP+PX)/2\right\rangle _{\chi
}-\left\langle X\right\rangle _{\chi}\left\langle P\right\rangle _{\chi}%
\]
is the quantum covariance. (The classic Heisenberg principle omits the
covariance term on the right-hand side of \eqref{schrUncertainty}.)

States that give equality in (\ref{schrUncertainty}) are Gaussian wave
packets, but where the quadratic term in the exponent can be complex, as
follows:%
\begin{equation}
\chi(x)=C\exp\left\{iax^{2}-b(x-c)^{2}+idx\right\} \label{coherentComplexTime}%
\end{equation}
with $a,b,c,d\in\mathbb{R}$ and $b>0$. This class of states is actually more
natural than the usual ones with $a=0$, because the collection of states of
the form (\ref{coherentComplexTime}) is invariant under the metaplectic
representation; that is, the natural (projective) unitary action of the group
of symplectic linear transformations of $\mathbb{R}^{2}$.

If we specialize the states in (\ref{coherentKc}) to the $\mathbb{R}^{d}$
case, we find that they are Gaussian wave packets, and that if $\mathrm{Im}%
\,\tau\neq0$ then the quadratic part of the exponent is complex. We see, then,
that allowing the time-parameter in the Segal--Bargmann transform to be
complex amounts to considering a larger and more natural family of coherent
states. In the $\mathbb{R}^{d}$ case, unitary Segal--Bargmann-type transforms
using general Gaussian wave packets were constructed by J. Sj\"ostrand
\cite{Sjostrand1996} and L. H\"ormander \cite{Hormander1991}, with
applications to semiclassical analysis. In these works, it is essential to
allow the quadratic part of the exponent to be complex, in order to achieve
invariance of the theory under symplectic linear transformations.

In the $s\rightarrow\infty$ transform $B_{\infty,t+iu}$ of Theorem
\ref{largeS.thm}, the domain Hilbert space is $L^{2}(K)$. Since $e^{iu\Delta
/2}$ is a unitary map of $L^{2}(K)$ to itself, in this case it is possible to
derive the complex-time transform from the real one $B_{\infty,t}$ (denoted as
the $C$-version of the transform $C_{t}$ in \cite{Hall1994}) by the
decomposition $e^{\frac{1}{2}(t+iu)\Delta}=e^{t\Delta/2}e^{iu\Delta/2}$. This
possibility has been exploited, for example, in the papers
\cite{Florentino2002,Florentino2003} of C. Florentino, J. Mour\~{a}o, and J.
Nunes on the quantization of nonabelian theta functions on $\mathrm{SL}%
(n,\mathbb{C})=\mathrm{SU}(n)_{\mathbb{C}}$. The authors show that these
functions arise as the image of certain distributions on $\mathrm{SU}(n)$
under the heat operator, evaluated at a complex time, and use the
Segal--Bargmann transform in the complexification process. These papers, then,
show the utility of introducing a complex time-parameter into the
($C$-version) Segal--Bargmann transform. The present paper extends this
complex time-parameter to the two-parameter transform.

Meanwhile, the Segal--Bargmann transform for $K$ is related to the study of
complex structures on the cotangent bundle $T^{\ast}(K)$. There is a natural
one-parameter family of \textquotedblleft adapted complex
structures\textquotedblright\ on $T^{\ast}(K)$ arising from a general
construction of Guillemen--Stenzel \cite{Guillemin1991,Guillemin1992} and
Lempert--Sz\H{o}ke \cite{Lempert1991,Szoke1991}. Motivated by ideas of
Thiemann \cite{Thiemann1996}, the second author and W. Kirwin in
\cite{Hall2011} showed that these structures arise from the \textquotedblleft
imaginary-time geodesic flow\textquotedblright\ on $T^{\ast}(K)$. The
Segal--Bargmann transform can then be understood
\cite{Florentino2005,Florentino2006,Hall2002} as a quantum counterpart of the
construction in \cite{Hall2011}.

As observed in \cite{Lempert2012}, the adapted complex structures on $T^{\ast
}(K)$ extend to a two-parameter family, by including both a real and an
imaginary part to the time-parameter in the geodesic flow in \cite{Hall2011}.
The corresponding quantum construction has been done in \cite{Lempert2015} and
can be thought of as adding a complex parameter to the $C$-version of the
Segal--Bargmann transform for $K$. (Compare work of Kirwin and Wu
\cite{Kirwin2006} in the $\mathbb{R}^{d}$ case.) The present paper then
extends the complex-time transform to its most natural range, in which the
domain Hilbert space is taken to be $L^{2}$ of $K$ with respect to a heat
kernel measure.

Finally, we mention the paper \cite{Hall2008}, which shows that certain
operators on $L^{2}(K_{\mathbb{C}},\nu_{t})$ of the form $C_{t}AC_{t}^{-1}$,
where $A$ is an operator on $L^{2}(K)$, can be represented as Toeplitz
operators. Here $C_{t}$, for $t\in\mathbb{R}$, is the $C$-version
Segal--Bargmann transform, which coincides with the limiting transform
$B_{\infty,t}$ in Theorem \ref{largeS.thm}. Using the results of the present
paper, a similar analysis can be performed for operators of the form
$C_{t+iu}AC_{t+iu}^{-1}$, where $C_{t+iu}$ is the limiting transform
$B_{\infty,t+iu}$ in Theorem \ref{largeS.thm}.

\section{Compact-Type Lie Groups and their Complexifications\label{sec.ctlgtc}%
\label{Section Complexification}}

We now introduce the class of Lie groups in which we are interested: those of
compact type and their complexifications.

\begin{definition}
\label{d.Ad-inv} A connected Lie group $K$ with Lie algebra $\mathfrak{k}$ is
said to be of \textbf{compact type} if there exists an $\mathrm{Ad}$%
-$K$-invariant inner product on $\mathfrak{k},$ that is, an inner product such
that%
\[
\left\langle \mathrm{Ad}_{x}X,\mathrm{Ad}_{x}Y\right\rangle =\left\langle
X,Y\right\rangle ,\quad\forall x\in K,~X,Y\in\mathfrak{k}.
\]

\end{definition}

Clearly a commutative group is of compact type.\ Furthermore, every compact
group is of compact type, since any inner product on its Lie algebra can be
made Ad-invariant by averaging over the adjoint action. A key result says that
products of these two examples account for all Lie groups of compact type.

\begin{proposition}
[\cite{Milnor76}, Lemma 7.5]\label{p.compact.type.decomp} If $K$ is a
compact-type Lie group with a specified $\mathrm{Ad}$-invariant inner product,
then $K$ is isometrically isomorphic to a direct product group: $K\cong
K_{0}\times\mathbb{R}^{d}$ for some compact Lie group $K_{0}$ and some
non-negative integer $d$.
\end{proposition}

If $G$ is a connected real Lie group, a \textbf{complexification} of $G$ is a
pair $(G_{\mathbb{C}},\iota)$ consisting of a complex Lie group $G_{\mathbb{C}%
}$ and a smooth homomorphism $\iota:G\rightarrow G_{\mathbb{C}}$ such that the
following universal property holds: For any complex Lie group $H$ and any
smooth homomorphism $\Phi:G\rightarrow H,$ there is a unique holomorphic
homomorphism $\Phi_{\mathbb{C}}:G_{\mathbb{C}}\rightarrow H$ such that
\[
\Phi_{\mathbb{C}}\circ\iota=\Phi.
\]

Suppose $K=K_{0}\times\mathbb{R}^{d}$ is a connected Lie group of compact
type. It is known (\cite[XVII Theorem 5.1]{Hochschild} or \cite[Theorem 4.1,
Propositions 8.4 and 8.6]{Brocker-tomDieck}) that the Lie algebra of the
complexification of $K_{0}$ is the complexification of its Lie algebra
$\mathfrak{k}_{0}$---that is, $\mathrm{Lie}((K_{0})_{\mathbb{C}}%
)=\mathfrak{k}+i\mathfrak{k}$---and that $\iota$ maps $K_{0}$ injectively into
its complexification. Meanwhile, the complexification of $\mathbb{R}^{d}$ is
$\mathbb{C}^{d},$ with $\iota$ being the obvious inclusion map. Thus, the Lie
algebra of $K_{\mathbb{C}}$ is the complexification of its Lie algebra, and
$\iota:K\rightarrow K_{\mathbb{C}}$ is injective. From now on, we always
identify $K$ with the subgroup $\iota(K)$ of $K_{\mathbb{C}}.$

\begin{example}
The compact Lie groups $\mathrm{SO}(n)$, $\mathrm{SU}(n)$, and $\mathrm{U}(n)$
have the following complexifications:
\[
\mathrm{SO}(n)_{\mathbb{C}}=\mathrm{SO}(n;\mathbb{C}),\quad\mathrm{SU}%
(n)_{\mathbb{C}}=\mathrm{SL}(n;\mathbb{C}),\quad\mathrm{U}(n)_{\mathbb{C}%
}=\mathrm{GL}(n;\mathbb{C}).
\]

\end{example}

We recall that a Lie group is called \textbf{unimodular} if every left Haar
measure is also right invariant.

\begin{proposition}
\label{p.compact.type.unimodular} If $K$ is a connected Lie group of compact
type, both $K$ and $K_{\mathbb{C}}$ are unimodular.
\end{proposition}

\begin{proof}
It is well known that a connected Lie group $G$ is unimodular if and only if
$\mathrm{Ad}_{g}$ has determinant equal to 1 for all $g\in G$. (See, for
example, Exercise 26 in Chapter 2 of \cite{Varadarajan}.) There is, however, a
subtlety: even if $G$ happens to be a complex Lie group, so that
$\mathrm{Ad}_{g}$ is a complex-linear map on the the Lie algebra, the
determinant of $\mathrm{Ad}_{g}$ should be taken over $\mathbb{R}$ not
$\mathbb{C}$.

Fix an Ad-invariant inner product on $\mathfrak{k}$ and identify
$\mathfrak{k}$ with $\mathbb{R}^{n}$ using an orthonormal basis. Then the
adjoint representation of $K$ maps into $\mathrm{O}(n)$, and actually into
$\mathrm{SO}(n)$, since $K$ is connected. Thus, $\mathrm{Ad}_{x}$ has
determinant equal to 1 for all $x\in K$, showing that $K$ is unimodular.

Now, since the adjoint representation of $K$ maps into $\mathrm{SO}(n)$, the
adjoint representation of $K_{\mathbb{C}}$ maps into the complexification of
$\mathrm{SO}(n)$, namely $\mathrm{SO}(n;\mathbb{C})$. It follows that for all
$g\in K_{\mathbb{C}}$, the determinant of $\mathrm{Ad}_{g}$, computed over
$\mathbb{C}$, is 1. It is then easily verified that the determinant over
$\mathbb{R}$ of a complex-linear transformation, viewed as a real linear
transformation, is the square of the absolute value of the determinant over
$\mathbb{C}$. (Work in a basis over $\mathbb{C}$ in which the operator is
upper triangular.) Thus, the determinant of $\mathrm{Ad}_{g}$, computed over
$\mathbb{R}$, is also 1, showing that $K_{\mathbb{C}}$ is unimodular.
\end{proof}

Let $K$ be a connected Lie group of compact type and $K_{\mathbb{C}}$ its
complexification. It is convenient, for reasons that will be apparent shortly,
to write the \textquotedblleft multiplication by $i$\textquotedblright\ map on
$\mathfrak{k}_{\mathbb{C}}$ as $J:\mathfrak{k}_{\mathbb{C}}\rightarrow
\mathfrak{k}_{\mathbb{C}}$. (Thus, $J^{2}=-I$.) Since $\mathfrak{k}%
_{\mathbb{C}}$ is a complex Lie algebra, the bracket on $\mathfrak{k}%
_{\mathbb{C}}$ is bilinear over $\mathbb{C}$, and in particular
\begin{equation}
\lbrack JX,Y]=J[X,Y] \label{Jbracket}%
\end{equation}
for all $X,Y\in\mathfrak{k}_{\mathbb{C}}$.

For any $X\in\mathfrak{k}_{\mathbb{C}}$, the left-invariant vector field
$\tilde{X}$ is given by%
\begin{equation}
(\tilde{X}f)(g)=\left.  \frac{d}{dt}f(ge^{tX})\right\vert _{t=0}
\label{equ.left}%
\end{equation}
for any smooth real- or complex-valued function $f$ on $K_{\mathbb{C}}$. We
may now appreciate the utility of the notion $J$ for the \textquotedblleft
multiplication by $i$\textquotedblright\ map on $\mathfrak{k}_{\mathbb{C}}$:
in general, $\widetilde{JX}f$ is not the same as $i\,\tilde{X}f$ (for example,
if $f$ is real valued). On the other hand, a complex-valued function $f$ on
$K_{\mathbb{C}}$ is holomorphic if and only if the differential of $f$ at each
point $g\in K_{\mathbb{C}}$ is a complex-linear map from $T_{g}(K_{\mathbb{C}%
})$ to $\mathbb{C}$. Thus, if $f$ is holomorphic, then for all $X\in
\mathfrak{g}$ and $g\in K_{\mathbb{C}}$, we have
\begin{equation}
\widetilde{JX}f(g)=i\tilde{X}f(g)\quad(f\text{ holomorphic}).
\label{Jholomorphic}%
\end{equation}

\section{Invariant Metrics on $K$ and $K_{\mathbb{C}}$
\label{section Inner Products}}

\subsection{Invariant metrics\label{invariantMetrics.sec}}

If $G$ is a Lie group with Lie algebra $\mathfrak{g}$ and $K\subseteq G$ is a
compact Lie subgroup, one can produce an $\mathrm{Ad}(K)$-invariant inner
product on $\mathfrak{g}$ by averaging any inner product over the adjoint
representation of $K$, as above. This raises the question: how many
$\mathrm{Ad}(K)$-invariant inner products does $G$ possess? We now answer this
question in the case that $K$ is simple (and compact type), and
$G=K_{\mathbb{C}}$ is the complexification of $K$.

Fix a compact-type Lie group $K$, and an $\mathrm{Ad}(K)$-invariant inner
product $\langle\cdot,\cdot\rangle_{\mathfrak{k}}$ on its Lie algebra
$\mathfrak{k}$. Let $K_{\mathbb{C}}$ denote the complexification of $K$
(cf.\ Section \ref{Section Complexification}); in particular $\mathfrak{k}%
_{\mathbb{C}}\equiv\mathrm{Lie}(K_{\mathbb{C}})=\mathfrak{k}\oplus
J\mathfrak{k}$. Consider the following three-parameter family of inner
products on $K_{\mathbb{C}}$:
\begin{equation}
\langle X_{1}+JY_{1},X_{2}+JY_{2}\rangle_{a,b,c}:=a\langle X_{1},X_{2}%
\rangle_{\mathfrak{k}}+b\langle Y_{1},Y_{2}\rangle_{\mathfrak{k}}+c(\langle
X_{1},Y_{2}\rangle_{\mathfrak{k}}+\langle X_{2},Y_{1}\rangle_{\mathfrak{k}})
\label{e.inner.products}%
\end{equation}
for $X_{1},X_{2},Y_{1},Y_{2}\in\mathfrak{k}$, where $a,b>0$ and $c^{2}<ab$. It
is straightforward to verify that the symmetric bilinear forms in
\eqref{e.inner.products} are real inner products on $\mathfrak{k}_{\mathbb{C}%
}$ (precisely under the conditions on $a,b,c$ stated below the equation), and
are all $\mathrm{Ad}(K)$-invariant. The main theorem of this section is that,
in the case that $K$ is simple, this is a complete characterization of all
$\mathrm{Ad}(K)$-invariant inner products on $K_{\mathbb{C}}$.

\begin{definition}
A Lie group $K$ is called \emph{simple} if $\mathrm{dim}\,K\ge2$, and the Lie
algebra $\mathfrak{k}$ of $K$ has no nontrivial ideals.
\end{definition}

\begin{theorem}
\label{t.Ad(K)-inv.inn.prod} Let $K$ be a simple (or $1$-dimensional) Lie
group of compact type. Then $\mathfrak{k}$ has a unique (up to scale)
$\mathrm{Ad}$-invariant real inner product $\langle\cdot,\cdot\rangle
_{\mathfrak{k}}$, and \emph{all} $\mathrm{Ad}(K)$-invariant real inner
products on $\mathfrak{k}_{\mathbb{C}}$ have the form \eqref{e.inner.products}.
\end{theorem}

\begin{remark}
For example, $K=\mathrm{SU}(n)$ is simple, with complexification
$K_{\mathbb{C}}=\mathrm{SL}(n,\mathbb{C})$. Hence \eqref{e.inner.products}
characterizes all $\mathrm{Ad}(\mathrm{SU}(n))$-invariant inner products on
$\mathrm{SL}(n,\mathbb{C})$, where $\langle X,Y\rangle_{\mathfrak{su}%
(n)}=\mathrm{Tr}(XY^{\ast}) = -\mathrm{Tr}(XY)$ is the unique (up to scale)
$\mathrm{Ad}$-invariant inner product on $\mathfrak{su}(n)$. In that case, the
family can be written explicitly in terms of the trace as
\begin{equation}
\langle A,B\rangle_{a,b,c}=\frac{1}{2}(b+a)\mathrm{Re}\,\mathrm{Tr}(AB^{\ast
})+\frac{1}{2}\mathrm{Re}\left[  (b-a+2ic)\mathrm{Tr}(AB)\right]  .
\label{e.inner.products.SL(n,C)}%
\end{equation}

Extending to $\mathrm{U}(n)$ and its complexification $\mathrm{GL}%
(n,\mathbb{C})$, it is easy to compute that all $\mathrm{Ad}(\mathrm{U}%
(n))$-invariant inner products on $\mathfrak{gl}(n,\mathbb{C})$ are of the
form \eqref{e.inner.products} plus one more term, involving the $1$%
-dimensional subspace spanned by the identity matrix; extending
\eqref{e.inner.products.SL(n,C)}, there is one more term involving
$\mathrm{Tr}(A)\mathrm{Tr}(B)$. In \cite{Kemp2015c,Kemp2015a,Kemp2015b}, the
third author studied the large-$n$ limits of the diffusion processes on
$\mathrm{GL}(n,\mathbb{C})$ invariant with respect to the inner products
$\langle\,,\,\rangle_{a,b,0}$. Part of the motivation for the present work was
the question of whether those were the largest class of appropriately
invariant diffusions; the answer provided by Theorem
\ref{t.Ad(K)-inv.inn.prod} is no.
\end{remark}

\begin{remark}
The first statement of Theorem \ref{t.Ad(K)-inv.inn.prod}, that the
$\mathrm{Ad}$-invariant inner product on $K$ is unique up to scale when $K$ is
simple, is well known; it was proved, for example, in \cite[Lemma
7.6]{Milnor76}.
\end{remark}

We will use Schur's lemma as a tool in the proof of Theorem
\ref{t.Ad(K)-inv.inn.prod}, but this is complicated by the fact that the inner
products in question are \emph{real}. We must therefore be careful about how
and when we complexify.

\begin{lemma}
\label{l.Ad.irred} If $K$ is a simple (real) Lie group with Lie algebra
$\mathfrak{k}$, then the (real) adjoint representation of $K$ on
$\mathfrak{k}$ is irreducible. Moreover, if $K$ is compact type, then the
(complex) adjoint representation of $K$ on $\mathfrak{k}_{\mathbb{C}}$ is also irreducible.
\end{lemma}

\begin{proof}
If $\mathcal{I}\subseteq\mathfrak{k}$ is an invariant real subspace for
$\mathrm{Ad}(K)$, then $\mathrm{Ad}_{e^{tX}}(Y)\in\mathcal{I}$ for all
$t\in\mathbb{R}$, $X\in\mathfrak{k}$, and $Y\in\mathcal{I}$. Taking the
derivative at $t=0$ shows that $\mathrm{ad}_{X}(Y) = [X,Y]\in\mathcal{I}$ for
all $X\in\mathfrak{k}$ and $Y\in\mathcal{I}$, which means $\mathcal{I}%
\subseteq\mathfrak{k}$ is an ideal in $\mathfrak{k}$. Thus $\mathcal{I}%
\in\{0,\mathfrak{k}\}$, yielding the first statement of the lemma.

Now, \cite[Theorem 7.32]{HallLieBook} states that the simplicity of
$\mathfrak{k}$ implies that $\mathfrak{k}_{\mathbb{C}}$ is also simple as a
complex Lie algebra. (The statement given there assumes $K$ is compact, but
the proof only uses the fact that it is compact type.) So, let $\mathcal{J}%
\subseteq\mathfrak{k}_{\mathbb{C}}$ be an invariant complex subspace for
$\mathrm{Ad}(K)$. The same argument above shows that $[X,W]\in\mathcal{J}$ for
all $X\in\mathfrak{k}$ and $W\in\mathcal{J}$. Any $Z\in\mathfrak{k}%
_{\mathbb{C}}$ has the form $Z=X+JY$ for $X,Y\in\mathfrak{k}$, and by
\eqref{Jbracket}, we therefore have
\[
[Z,W] = [X+JY,W] = [X,W]+J[Y,W]\in\mathcal{J}+J\mathcal{J}=\mathcal{J},
\qquad\forall\; Z\in\mathfrak{k}_{\mathbb{C}}, W\in\mathcal{J}
\]
where the final equality follows from the fact that $\mathcal{J}$ is a
\emph{complex} subspace. Hence $\mathcal{J}$ is a complex ideal in
$\mathfrak{k}_{\mathbb{C}}$, and therefore $\mathcal{J}\in\{0,\mathfrak{k}%
_{\mathbb{C}}\}$. This concludes the proof of the second statement.
\end{proof}

We now prove the algebraic result that constitutes most of the proof of
Theorem \ref{t.Ad(K)-inv.inn.prod}.

\begin{proposition}
\label{p.Ad(K)-inv.bilinear} Let $K$ be a simple (or $1$-dimensional) real
compact-type Lie group, and fix an $\mathrm{Ad}$-invariant inner product
$\langle\cdot,\cdot\rangle_{\mathfrak{k}}$ on its Lie algebra $\mathfrak{k}$.
If $\mathcal{B}\colon\mathfrak{k}_{\mathbb{C}}\times\mathfrak{k}_{\mathbb{C}%
}\to\mathbb{R}$ is an $\mathrm{Ad}(K)$-invariant symmetric bilinear form, then
$\mathcal{B}$ has the form \eqref{e.inner.products} for some $a,b,c\in
\mathbb{R}$.
\end{proposition}

\begin{proof}
The result is straightforward when $K$ is $1$-dimensional, so we focus on the
case that $K$ is simple. We use the inner product $\langle\cdot,\cdot
\rangle_{1,1,0}$ (cf.\ \eqref{e.inner.products}) as a reference; there is then
some endomorphism $M\colon\mathfrak{k}_{\mathbb{C}}\rightarrow\mathfrak{k}%
_{\mathbb{C}}$ such that
\[
\mathcal{B}(Z,W)=\langle Z,M(W)\rangle_{1,1,0}\qquad\forall\;Z,W\in
\mathfrak{k}_{\mathbb{C}}.
\]
The symmetry of $\mathcal{B}$ forces $M$ to be self-adjoint. We identify
$\mathfrak{k}_{\mathbb{C}}=\mathfrak{k}\oplus J\mathfrak{k}$ with
$\mathfrak{k}\oplus\mathfrak{k}$. Thus we can decompose the endomorphism $M$
in block diagonal form
\begin{equation}
M=\left[
\begin{array}
[c]{cc}%
A & C\\
C^{\top} & B
\end{array}
\right]  \label{e.block.matrix.M}%
\end{equation}
where $A$ and $B$ are symmetric matrices.

Since the adjoint representation of $K$ commutes with $J$, it follows that,
under the isomorphism $\mathfrak{k}_{\mathbb{C}}\cong\mathfrak{k}%
\oplus\mathfrak{k}$, $\mathrm{Ad}_{k}$ acts diagonally for all $k\in K$. Using
the fact that both the inner product $\langle\cdot,\cdot\rangle_{1,1,0}$ and
the bilinear form $\mathcal{B}$ are $\mathrm{Ad}_{k}$-invariant, it is
straightforward to compute that the matrices $A$, $B$, $C$, and $C^{\top}$ all
commute with $\mathrm{Ad}_{k}$ for each $k\in K$. The same therefore applies
to the complex-linear extensions of these endomorphisms to $\mathfrak{k}%
_{\mathbb{C}}$. It then follows from Lemma \ref{l.Ad.irred} and Schur's lemma
that there are constants $a,b,c\in\mathbb{C}$ with $A=aI$, $B=bI$, and
$C=C^{\top}=cI$. Since each of the endomorphisms preserves the real subspace
$\mathfrak{k}$, it follows that $a,b,c\in\mathbb{R}$.

Hence, for $Z=X+JY\in\mathfrak{k}_{\mathbb{C}}$, \eqref{e.block.matrix.M}
yields $M(Z) =(aX+cY)+J(cX+bY)$. From the definition of the inner product
$\langle\cdot,\cdot\rangle_{1,1,0}$, we therefore have
\begin{align*}
\mathcal{B}(X_{1}+JY_{1},X_{2}+JY_{2})  &  = \langle X_{1}+JY_{1}%
,aX_{2}+cY_{2}+J(cX_{2}+bY_{2})\rangle_{1,1,0}\\
&  = \langle X_{1},aX_{2}+cY_{2}\rangle_{\mathfrak{k}} + \langle Y_{1}%
,cX_{2}+bY_{2}\rangle_{\mathfrak{k}}\\
&  = a\langle X_{1},X_{2}\rangle_{\mathfrak{k}}+c\langle X_{1},Y_{2}%
\rangle_{\mathfrak{k}}+c\langle Y_{1},X_{2}\rangle_{\mathfrak{k}}+b\langle
Y_{1},Y_{2}\rangle_{\mathfrak{k}}\\
&  = \langle X_{1}+JY_{1},X_{2}+JY_{2}\rangle_{a,b,c}%
\end{align*}
concluding the proof.
\end{proof}

The proof of Theorem \ref{t.Ad(K)-inv.inn.prod} now follows quite easily.

\begin{proof}
[Proof of Theorem \ref{t.Ad(K)-inv.inn.prod}]Let $\langle\cdot,\cdot
\rangle_{\mathfrak{k}}$ and $\langle\cdot,\cdot\rangle_{\mathfrak{k}}^{\prime
}$ denote two $\mathrm{Ad}$-invariant inner products on $K$. We may view the
second inner product as a symmetric (degenerate) bilinear form on
$\mathfrak{k}_{\mathbb{C}}$, which is $\mathrm{Ad}(K)$-invariant. By
Proposition \ref{p.Ad(K)-inv.bilinear}, it follows that $\langle\cdot
,\cdot\rangle_{\mathfrak{k}}^{\prime}= a\langle\cdot,\cdot\rangle
_{\mathfrak{k}}$ for some $a\in\mathbb{R}$ (the other terms in
\eqref{e.inner.products} are $0$); the fact that both are inner products
forces $a>0$. This proves the uniqueness, up to scale, of the $\mathrm{Ad}%
$-invariant inner product on $K$.

Now, any real inner product $\langle\cdot,\cdot\rangle$ on $\mathfrak{k}%
_{\mathbb{C}}$ is a symmetric bilinear form on $\mathfrak{k}_{\mathbb{C}}$,
and so by $\mathrm{Ad}(K)$-invariance, Proposition \ref{p.Ad(K)-inv.bilinear}
shows that it has the form \eqref{e.inner.products} for some $a,b,c\in
\mathbb{R}$. Since it is an inner product, it follows that the matrix $M$ of
\eqref{e.block.matrix.M} is positive definite, and given its block diagonal
form, this is equivalent to $a,b>0$ and $ab-c^{2}>0$. This concludes the proof.
\end{proof}

\subsection{Laplacians\label{laplacians.sec}}

We use the notation $\tilde{X}$ for the left-invariant vector field associated
to a Lie algebra element $X,$ as in (\ref{equ.left}). We fix an $\mathrm{Ad}%
(K)$-invariant inner product $\left\langle \cdot,\cdot\right\rangle
_{\mathfrak{k}}$ on $K.$ Then if $\{X_{j}\}_{j=1}^{\dim\mathfrak{k}}$ is an
orthonormal basis for $\mathfrak{k}$ with respect to $\left\langle \cdot
,\cdot\right\rangle _{\mathfrak{k}},$ we define $\Delta_{K}$ to be the
operator given by%
\begin{equation}
\Delta_{K}=\sum_{j=1}^{\dim\mathfrak{k}}\tilde{X}_{j}^{2}.
\label{LaplacianOnK}%
\end{equation}
The operator is easily seen to be independent of the choice of orthonormal
basis. Since $K$ is unimodular, this operator is the Laplace--Beltrami
operator for the left-invariant metric determined by $\left\langle \cdot
,\cdot\right\rangle _{\mathfrak{k}}$ \cite[Remark 2.2]{Driver1997}. Since
$\left\langle \cdot,\cdot\right\rangle _{\mathfrak{k}}$ is $\mathrm{Ad}%
(K)$-invariant, the metric on $K$ is actually bi-$K$-invariant and thus
$\Delta_{K}$ is bi-$K$-invariant.

We now fix real numbers $a,$ $b,$ and $c$ with $a,b>0$ and $c^{2}<ab,$ as in
Section \ref{invariantMetrics.sec}, and let $\left\langle \cdot,\cdot
\right\rangle _{a,b,c}$ be the associated $\mathrm{Ad}(K)$-invariant inner
product. We then choose an orthonormal basis $\{Z_{j}\}_{j=1}^{2\dim
\mathfrak{k}}$ for $\mathfrak{k}_{\mathbb{C}}$ with respect to this inner
product and define the Laplacian $L_{a,b,c}$ by%
\begin{equation}
L_{a,b,c}=\sum_{j=1}^{2\dim\mathfrak{k}}Z_{j}^{2}, \label{LaplacianOnKc}%
\end{equation}
similarly to (\ref{LaplacianOnK}).

\begin{proposition}
\label{p.Labc} Let $L_{a,b,c}$ denote the Laplacian in (\ref{LaplacianOnKc}).
Fix any basis $\{X_{j}\}_{j=1}^{d}$ of $\mathfrak{k}$ orthonormal with respect
to the given $\mathrm{Ad}(K)$-invariant inner product on $\mathfrak{k}$, and
let $Y_{j}=JX_{j}$. Then
\begin{equation}
L_{a,b,c}=\frac{1}{ab-c^{2}}\sum_{j=1}^{d}\left[  b\tilde{X}_{j}^{2}%
+a\tilde{Y}_{j}^{2}-2c\tilde{X}_{j}\tilde{Y}_{j}\right]  . \label{e.Labc}%
\end{equation}

\end{proposition}

\begin{proof}
We use the basis $\{Z_{j}\}_{j=1}^{2\dim\mathfrak{k}}$ consisting of
$X_{1},Y_{1},\ldots,X_{\dim\mathfrak{k}},Y_{\dim\mathfrak{k}}$ (in that
order). We let $\{q_{lm}\}_{l,m=1}^{2\dim\mathfrak{k}}$ be the associated Gram
matrix, that is, the matrix of inner products of these basis elements with
respect to the inner product $\left\langle \cdot,\cdot\right\rangle _{a,b,c}.$
If $q^{-1}$ is the inverse matrix to $q,$ it is an elementary computation to
verify that
\begin{equation}
L_{a,b,c}=\sum_{l,m=1}^{2\dim\mathfrak{k}}(q^{-1})_{lm}\tilde{Z}_{l}\tilde
{Z}_{m}. \label{e.Labc.0}%
\end{equation}

Now, we can compute directly from (\ref{e.inner.products}) and the
orthonormality of $\{X_{j}\}_{j=1}^{d}$ that
\[
\langle X_{i},X_{j}\rangle_{a,b,c}=a\delta_{ij},\quad\langle Y_{i}%
,Y_{j}\rangle_{a,b,c}=b\delta_{ij},\quad\langle X_{i},Y_{j}\rangle
_{a,b,c}=\langle Y_{i},X_{j}\rangle_{a,b,c}=c\delta_{ij}.
\]
It follows that the matrix $q$ is block diagonal with $2\times2$ diagonal
blocks all equal to the matrix $B$ (below). Thus $q^{-1}$ is also block
diagonal with $2\times2$ diagonal blocks all equal to $B^{-1}$ (below):
\[
B=\left[
\begin{array}
[c]{cc}%
a & c\\
c & b
\end{array}
\right]  ,\qquad B^{-1}=\frac{1}{ab-c^{2}}\left[
\begin{array}
[c]{cc}%
b & -c\\
-c & a
\end{array}
\right]  .
\]
Combining this with \eqref{e.Labc.0} yields \eqref{e.Labc}.
\end{proof}

To dispense with the cumbersome determinant in the denominator in
\eqref{e.Labc}, and match the parametrization relevant to the Segal--Bargmann
transform, we make the following change of parametrization:
\begin{equation}
(s,t,u)=\Phi(a,b,c):=\frac{1}{ab-c^{2}}(a+b,2a,2c). \label{e.(a,b,c)->(s,t,u)}%
\end{equation}
It is straightforward to verify that $\Phi$ is a diffeomorphism
\[
\Phi\colon\{(a,b,c)\colon a,b>0,c^{2}<ab\}\rightarrow\{(s,t,u)\colon
t>0,u\in\mathbb{R},2s>t+u^{2}/t\}
\]
with inverse
\begin{equation}
(a,b,c)=\Phi^{-1}(s,t,u)=\frac{4}{2st-t^{2}-u^{2}}(\textstyle{\frac{t}%
{2},s-\frac{t}{2},\frac{u}{2}})=\frac{1}{\alpha}(\textstyle{\frac{t}%
{2},s-\frac{t}{2},\frac{u}{2}}) \label{e.Phi.inverse.formula}%
\end{equation}
referring to the constant $\alpha$ of \eqref{e.alpha}, which is positive
precisely in range of $\Phi$. From here on, we use the parameters $(s,t,u)$
which leads to the notation used in Definition \ref{d.Lap.rsu} of
$\Delta_{s,\tau}$ on $K_{\mathbb{C}}$ in the introduction. In particular, this
means that the Laplacian $\Delta_{s,\tau}$ corresponds to the inner product
$\langle\cdot,\cdot\rangle_{a,b,c}$ where $(a,b,c)$ are given as in
\eqref{e.Phi.inverse.formula}. The fact that $\Phi$ is a bijection shows that
there is a one-to-one correspondence between the Laplacians $\Delta_{s,\tau}$
and the inner products $\langle\cdot,\cdot\rangle_{a,b,c}$.

\section{Heat kernels and matrix entries}

We refer the reader to \cite{Robinson1991} or \cite{Varopoulos1992} for the
general theory of heat kernels on Lie groups.

\subsection{Heat kernels on $K$ and $K_{\mathbb{C}}$}

We now fix a connected Lie group $K$ of compact type, together with an
$\mathrm{Ad}(K)$-invariant inner product $\left\langle \cdot,\cdot
\right\rangle _{\mathfrak{k}}$ on $\mathfrak{k}.$ We let $\Delta_{K}$ be the
associated Laplacian on $K,$ as in Section \ref{laplacians.sec}. We then let
$\rho_{t}$ be the associated \textbf{heat kernel} on $K,$ i.e., the
fundamental solution at the identity to the heat equation%
\[
\frac{\partial u}{\partial t}=\frac{1}{2}\Delta_{K}u.
\]
Then the heat operator may be computed as%
\begin{equation}
(e^{t\Delta/2}f)(x)=\int_{K}\rho_{t}(xy^{-1})f(y)~dy, \label{heatOpOnK}%
\end{equation}
where $dy$ is the Riemannian volume measure associated to the left-invariant
Riemannian metric on $K$ induced by the inner product $\left\langle
\cdot,\cdot\right\rangle _{\mathfrak{k}}$ on $\mathfrak{k}.$

\begin{remark}
For a general left-invariant metric on $K,$ the right-hand side of
(\ref{heatOpOnK}) should have $\rho_{t}(y^{-1}x)$ rather than $\rho
_{t}(xy^{-1}).$ Since however, our metric is $\mathrm{Ad}(K)$-invariant, the
heat kernel $\rho_{t}$ is a class function, so that $\rho_{t}(y^{-1}%
x)=\rho_{t}(xy^{-1}).$ We write $\rho_{t}(xy^{-1})$ to maintain consistency
with \cite{Hall1994}.
\end{remark}

We fix $s>0$ and $\tau\in\mathbb{C}$ with $\tau\in\mathbb{D}(s,s)$ (the disk
of radius $s$, centered at $s$). We consider a left-invariant metric on
$K_{\mathbb{C}}$ whose value at the identity is one of the inner products
considered in Section \ref{invariantMetrics.sec}. The associated Laplacian,
denoted $\Delta_{s,\tau},$ is the one considered in Definition \ref{d.Lap.rsu}%
. We emphasize that, although $\tau$ is a complex number, the Laplacian
$\Delta_{s,\tau}$ is a real elliptic operator on $K_{\mathbb{C}}.$ We then let
$\mu_{s,\tau,r}$ be the associated heat kernel, i.e., the fundamental solution
at the identity to the heat equation%
\[
\frac{\partial u}{\partial r}=\frac{1}{2}\Delta_{s,\tau}u,
\]
with $r$ being the time-variable in the heat equation. We will mainly be
interested in the value of this heat kernel at $r=1$:%
\[
\mu_{s,\tau}:=\mu_{s,t,1}.
\]
That is to say, formally,%
\[
\mu_{s,\tau}=e^{\Delta_{s,\tau}/2}(\delta),
\]
where $\delta$ is a $\delta$-function at the identity.

\begin{lemma}
[Averaging Lemma]\label{averaging.lem}Assume $K$ is compact. For each $s$ and
$\tau$ with $\tau\in\mathbb{D}(s,s)$, let $\nu_{s,\tau}$ be the associated
$K$\textbf{-averaged heat kernel}, given in Definition
\ref{d.K-av.heat.kernel}:
\[
\nu_{s,\tau}(g)=\int_{K}\mu_{s,\tau}(gk)~dk.
\]
Then there exist constants $a_{s,\tau}$ and $b_{s,\tau}$ such that%
\[
a_{s,\tau}\nu_{s,\tau}(g)\leq\mu_{s,\tau}(g)\leq b_{s,\tau}\nu_{s,\tau}(g)
\]
for all $g\in G.$ To be more precise: for each $s$ and $\tau,$ let $\sigma$ be any
positive number such that $\tau\in\mathbb{D}(s-\sigma,s-\sigma).$ Then we may
take%
\[
a_{s,\tau}=\min_{k\in K}\rho_{\sigma}(k);\quad b_{s,\tau}=\max_{k\in K}%
\rho_{\sigma}(k).
\]

\end{lemma}

\begin{proof}
We write the operator $\Delta_{s,\tau},$ as defined in
(\ref{e.def.Delta.s.tau}), in the form%
\begin{equation}
\Delta_{s,\tau}=\sigma\sum_{j=1}^{\dim\mathfrak{k}}\tilde{X}_{j}^{2}%
+\Delta_{s-\sigma,\tau}. \label{decomposeDelta}%
\end{equation}
Now, the operator $\Delta_{s-\sigma,\tau}$ is constructed from left-invariant
vector fields and is therefore a left-invariant operator on $K_{\mathbb{C}}.$
Since the inner product in the construction of $\Delta_{s,\tau}$ is
$\mathrm{Ad}(K)$-invariant, $\Delta_{s,\tau}$ is also invariant under the
\textit{right} action of $K.$ It follows that $\Delta_{s,\tau}$ commutes with
the left-invariant vector field $\tilde{X}$ on $K_{\mathbb{C}},$ with
$X\in\mathfrak{k},$ since $\tilde{X}$ is an infinitesimal right translation.
We conclude that the two terms on the right-hand side of (\ref{decomposeDelta}%
) commute. Once this observation has been made, the proof of the averaging
lemma from \cite[Lemma 11]{Hall1994} tells us that%
\[
\mu_{s,\tau}(g)=\int_{K}\mu_{s-\sigma,\tau}(gk^{-1})\rho_{\sigma}(k)~dk.
\]
Since $\rho_{s}(k)~dk$ is a probability measure, the integral of $\mu_{s,\tau
}$ over each $K$-orbit is the same as the corresponding integral of
$\mu_{s-\sigma,\tau}.$ Thus, we obtain%
\[
\mu_{s,\tau}(g)\leq\max_{k\in K}\rho_{\sigma}(k)\int_{K}\mu_{s-\sigma,\tau
}(gk^{-1})~dk=\max_{k\in K}\rho_{\sigma}(k)\nu_{s,\tau}(g),
\]
as claimed, and similarly for the lower bound.
\end{proof}

\subsection{Matrix entries\label{Section Matrix Entries}}

In the case $K=\mathbb{R}^{d}$, it is convenient to do computations with the
heat operator on polynomials. Although these functions are not in
$L^{2}(\mathbb{R}^{d})$, one can na\"{\i}vely make sense of $e^{\frac{t}%
{2}\Delta_{\mathbb{R}^{d}}}f$ as a terminating power series for any polynomial
$f$. It is then an easy matter to verify that the integral formula for the
heat operator coincides with its Taylor series. That is to say, if $f$ is a
polynomial on $\mathbb{R}^{d}$, then
\begin{equation}
\int_{\mathbb{R}^{d}}\rho_{t}(x-y)f(y)\,dy=\sum_{n=0}^{\infty}\frac{(t/2)^{n}%
}{n!}(\Delta_{\mathbb{R}^{d}})^{n}f(x). \label{e.heat.op.poly}%
\end{equation}
Equation \eqref{e.heat.op.poly} is easy to prove directly; the result is also
a special case of Proposition \ref{p.heat.op.on.m.e.} below.

We will need a counterpart of polynomial functions on a general (compact-type)
Lie group; these are \emph{matrix entries}, which we define as follows.

\begin{definition}
\label{d.matrix.entries} Let $G$ be a Lie group. Let $(\pi,V_{\pi})$ be a
finite-dimensional complex representation of $G$, and let $A\in\mathrm{End}%
(V_{\pi})$ be a fixed endomorphism. The associated \textbf{matrix entry}
function $f_{\pi,A}$ on $G$ is the function
\[
f_{\pi,A}(x)=\mathrm{Tr}(\pi(x)A).
\]
If $G$ is a complex Lie group and the representation $\pi:G\rightarrow
GL(V_{\pi})$ is holomorphic, then we refer to $f_{\pi,A}$ as a
\textbf{holomorphic matrix entry}. In particular, every holomorphic matrix
entry on a complex Lie group is a holomorphic function.
\end{definition}

\begin{remark}
\label{rk.m.e.lin.funct.}A number of comments on matrix entries are in order.

\begin{enumerate}
\item Although some authors might require $\pi$ to be irreducible in order to
call $f_{\pi,A}$ a \emph{matrix entry}, we make no irreducibility assumption
in our definition. If $G$ is compact, every finite-dimensional representation
of $G$ decomposes as a direct sum of irreducibles, in which case every matrix
entry is a linear combination of matrix entries for irreducible
representations. In general, not every matrix entry (in the sense of
Definition \ref{d.matrix.entries}) will decompose as a sum of matrix entries
of irreducible representations.

\item Some authors require a matrix entry to be of the form $f(x)=\xi
(\pi(x)v)$ for some $v\in V$ and $\xi\in V^{\ast}$. This is a special case of
Definition \ref{d.matrix.entries} with $f=f_{\pi,A}$ where $A(w)=\xi(w)v$,
i.e.,\ $A=\xi\otimes v$. The more general matrix entries of Definition
\ref{d.matrix.entries} are linear combinations of these more restricted
\textquotedblleft rank-1 type\textquotedblright\ entries.

\item Matrix entries are smooth functions on $G$.

\item If $G=\mathbb{R}^{d}$, all polynomials are matrix entries. Indeed: if
$q$ is a polynomial of degree $\leq n$, take the representation space $V$ to
be all polynomials $p$ of degree $\leq n$, where $\pi(x)p=p(\,\cdot+x)$. If
$\xi_{0}(p)=p(0)$ is the evaluation linear functional, then $\xi_{0}%
(\pi(x)q)=q(x)$, so $q$ is a matrix entry.

\item Even if $G$ is complex, we will have a reason to consider matrix entries
associated to representations of $G$ that are not holomorphic.
\end{enumerate}
\end{remark}

\begin{lemma}
\label{l.m.e.s.a.alg} For any Lie group $G$, the set of matrix entries on $G$
forms a self-adjoint complex algebra.
\end{lemma}

\begin{proof}
It is straightforward to compute that, for $\lambda\in\mathbb{C}$, $\lambda
f_{\pi,A}=f_{\pi,\lambda A}$, while sums and products satisfy $f_{\pi
,A}+f_{\sigma,B}=f_{\pi\oplus\sigma,A\oplus B}$ and $f_{\pi,A}f_{\sigma
,B}=f_{\pi\otimes\sigma,A\otimes B}$. For complex conjugation, we must define
the complex conjugate of a representation and an endomorphism. This can be
done invariantly, but for our purposes there is no reason not to simply choose
a basis. Given a representation $(\pi,V_{\pi})$ of dimension $d$, choose a
complex-linear isomorphism $\varphi\colon V_{\pi}\rightarrow\mathbb{C}^{d}$,
and let $[\pi(x)]=\varphi\circ\pi(x)\circ\varphi^{-1}$ and $[A]=\varphi\circ
A\circ\varphi^{-1}$. As $d\times d$ complex matrices, both $[\pi(x)]$ and
$[A]$ have complex conjugates $\overline{[\pi(x)]}$ and $\overline{[A]}$,
defined entry-wise. Then
\begin{equation}
\bar{f}_{\pi,A}(x)=\overline{\mathrm{Tr}(\pi(x)A)}=\overline{\mathrm{Tr}%
([\pi(x)][A])}=\mathrm{Tr}(\overline{[\pi(x)]}\,\overline{[A]}).
\label{e.conj.rep}%
\end{equation}
The map $\overline{[\pi]}\colon G\rightarrow\mathrm{GL}(\mathbb{C}^{d})$ given
by $\overline{[\pi]}(x)=\overline{[\pi(x)]}$ is a representation of $G$ on
$\mathbb{C}^{d}$, and \eqref{e.conj.rep} shows that
\[
\bar{f}_{\pi,A}=f_{\overline{[\pi]},\overline{A}}%
\]
is also a matrix entry of $G$. This concludes the proof.
\end{proof}

We now establish two key result about matrix entries.

\begin{theorem}
\label{t.HL^2.density} Let $K$ be a real Lie group of compact type. For any
$s>0$, the matrix entries on $K$ are dense in $L^{2}(K,\rho_{s})$. If $s>0$
and $\tau\in\mathbb{D}(s,s)$, then the holomorphic matrix entries on
$K_{\mathbb{C}}$ are dense in $\mathcal{H}L^{2}(K_{\mathbb{C}},\mu_{s,\tau})$.
\end{theorem}

\begin{proof}
We consider first the case that $K=\mathbb{R}^{d}$ and $K_{\mathbb{C}%
}=\mathbb{C}^{d}$. Then $\rho_{s}$ is a Gaussian measure on $K$. Since every
polynomial on $\mathbb{R}^{d}$ is a matrix entry, we may appeal to the
classical result that polynomials are dense in $L^{2}$ of Gaussian measures on
$\mathbb{R}^{d}$. (For a proof of a more general result, see \cite[Theorem
3.6]{Driver1999}.) On the complex side, every holomorphic polynomial is a
holomorphic matrix entry, and the measure $\mu_{s,\tau}$ on $\mathbb{C}^{d}$
is Gaussian. Thus, by \cite[Proposition 3.5]{Driver1999}, matrix entries are
dense in $\mathcal{H}L^{2}(\mathbb{C}^{d},\mu_{s,\tau})$. (Note that, in
general, the measure $\mu_{s,\tau}$ is not invariant under multiplication by
$e^{i\theta}$ and monomials of different degrees are not necessarily
orthogonal. Thus the proof of density of holomorphic polynomials in
\cite[Section 1b]{Bargmann1961} does not apply.)

We consider next the case that $K$ is compact. In that case, the heat kernel
density $\rho_{s}$ on $K$ is bounded and bounded away from zero for each fixed
$s>0$. Thus, the Hilbert space $L^{2}(K,\rho_{s})$ is the same as the Hilbert
space $L^{2}(K)$, with a different but equivalent norm. Hence, the density of
matrix entries in $L^{2}(K,\rho_{s})$ follows from the Peter--Weyl theorem. On
the complex side, we appeal to the averaging lemma (Lemma \ref{averaging.lem}%
), which tells us that the Hilbert space $\mathcal{H}L^{2}(K_{\mathbb{C}}%
,\mu_{s,\tau})$ is the same as the Hilbert space $\mathcal{H}L^{2}%
(K_{\mathbb{C}},\nu_{t})$, with a different but equivalent norm. Thus, it
suffices to establish the density of matrix entries in $\mathcal{H}%
L^{2}(K_{\mathbb{C}},\nu_{t})$; this claim follows verbatim from the proof of
the \textquotedblleft onto\textquotedblright\ part of Theorem 2 in
\cite[Section 8]{Hall1994}.

We consider finally the case of a general compact-type group $K$. Recall
(Proposition \ref{p.compact.type.decomp}) that $K$ is isometrically isomorphic
to $K_{0}\times\mathbb{R}^{d}$ for some compact Lie group $K_{0}$ and some
$d\geq0$. Thus, the heat kernel measure $\rho_{s}$ on $K$ factors as a product
of the heat kernel measures $\rho_{s}^{0}$ on $K_{0}$ and $\rho_{s}^{1}$ on
$\mathbb{R}^{d}$. Now, a standard result from measure theory tells us that
there is a unitary map $U$ from $L^{2}(K_{0},\rho_{s}^{0})\otimes
L^{2}(\mathbb{R}^{d},\rho_{s}^{1})$ onto $L^{2}(K,\rho_{s})$ uniquely
determined by the requirement that $U(f_{1}\otimes f_{2})(x_{1},x_{2}%
)=f_{1}(x_{1})f_{2}(x_{2})$. If $f_{1}$ and $f_{2}$ are matrix entries on
$K_{0}$ and $\mathbb{R}^{d}$, respectively, then $f_{1}(x_{1})f_{2}(x_{2})$ is
a matrix entry on $K$ (by an argument very similar to the proof of Lemma
\ref{l.m.e.s.a.alg}). Using the density results for $K_{0}$ and for
$\mathbb{R}^{d}$ and the unitary map $U$, we can easily show that linear
combinations of matrix entries of this sort (which are again matrix entries)
are dense in $L^{2}(K,\rho_{s})$.

On the complex side, $K_{\mathbb{C}}$ is isomorphic to $(K_{0})_{\mathbb{C}%
}\times\mathbb{C}^{d}$. If we restrict our Ad-invariant inner product on
$\mathfrak{k}$ to the Lie algebras of $K_{0}$ and of $\mathbb{R}^{d}$, these
restrictions will also be Ad-invariant. We may then construct left-invariant
metrics on $(K_{0})_{\mathbb{C}}$ and $\mathbb{C}^{d}$ by the same procedure
as for $K_{\mathbb{C}}$. In that case, it is easily verified that the
isomorphism $K_{\mathbb{C}}\cong(K_{0})_{\mathbb{C}}\times\mathbb{C}^{d}$ is
isometric. Thus, the heat kernel measure $\mu_{s,\tau}$ on $K_{\mathbb{C}}$ is
a product of the associated heat kernel measures $\mu_{s,\tau}^{0}$ on
$(K_{0})_{\mathbb{C}}$ and $\mu_{s,\tau}^{1}$ on $\mathbb{C}^{d}$.

Then, as on the real side, we have a unitary map $V$ from $L^{2}%
((K_{0})_{\mathbb{C}},\mu_{s,\tau}^{0})\otimes L^{2}(\mathbb{C}^{d}%
,\mu_{s,\tau}^{1})$ onto $L^{2}(K_{\mathbb{C}},\mu_{s,\tau})$. According to
the Appendix of \cite{Hall2001b}, the restriction of $V$ to the tensor product
of the two $\mathcal{H}L^{2}$ spaces maps \textit{onto} $\mathcal{H}%
L^{2}(K_{\mathbb{C}},\mu_{s,\tau})$. (It is easy to see that $V$ maps the
tensor product of the two $\mathcal{H}L^{2}$ spaces \textit{into}
$\mathcal{H}L^{2}(K_{\mathbb{C}},\mu_{s,\tau})$; it requires some small
argument to show that it maps onto.) Thus, as on the real side, the density
result for $K_{\mathbb{C}}$ reduces to the previously established results for
$(K_{0})_{\mathbb{C}}$ and for $\mathbb{C}^{d}$.
\end{proof}

\begin{proposition}
\label{p.heat.op.on.m.e.} Let $f_{\pi,A}$ be a matrix entry on $K.$ Then%
\begin{align}
\int_{K}\rho_{t}(xy^{-1})f_{\pi,A}(x)~dx  &  =\sum_{n=0}^{\infty}\frac{t^{n}%
}{2^{n}n!}(\Delta_{K})^{n}f_{\pi,A}(x)\nonumber\\
&  =\mathrm{Tr}(\pi(x)e^{tC_{\pi}/2}A) \label{e.heat.op.power.series.1}%
\end{align}
with absolute convergence of the integral on the left-hand side and locally
uniform convergence of the sums on the right-hand side. Here $C_{\pi}%
=\sum_{j=1}^{\dim\mathfrak{k}}\pi_{\ast}(X_{j})^{2},$ where $\pi_{\ast}$ is
the Lie algebra representation associated to the Lie group representation
$\pi.$

Let $f_{\pi,A}$ be a matrix entry on $K_{\mathbb{C}}.$ Then%
\begin{align}
\int_{K_{\mathbb{C}}}f_{\pi,A}(g)\mu_{s,\tau}(g)~dg  &  =\sum_{n=0}^{\infty
}\frac{1}{2^{n}n!}(\Delta_{s,\tau})^{n}f_{\pi,A}(e),\nonumber\\
&  =\mathrm{Tr}(\pi(x)e^{D_{\pi,s,\tau}/2}A) \label{e.heat.op.power.series.2}%
\end{align}
with absolute convergence of the integral on the left-hand side and locally
uniform convergence of the sum on the right-hand side. Here%
\[
D_{\pi,s,\tau}=\sum_{j=1}^{\dim\mathfrak{k}}\left[  \left(  s-\frac{t}%
{2}\right)  \pi_{\ast}(X_{j})^{2}+\frac{t}{2}\pi_{\ast}(Y_{j})^{2}%
-u\,\pi_{\ast}(X_{j})\pi_{\ast}(Y_{j})\right]
\]
where $\pi_{\ast}$ is the Lie algebra representation associated to the Lie
group representation $\pi.$
\end{proposition}

We note that unless $K$ is compact (as opposed to merely being of compact
type), matrix entries on $K$ are typically not in $L^{2}(K,dx)$ and thus not
in the usual domain of definition of the heat operator $e^{t\Delta_{K}/2}$.
Similarly, matrix entries on $K_{\mathbb{C}}$ are typically not in the usual
domain of the heat operator $e^{\Delta_{s,\tau}/2}$.

\begin{proof}
The proposition is an immediate consequence of Langland's theorem
(cf.\ \cite[Theorem 2.1]{Robinson1991}). See also \cite[Lemma 8]{Hall1994}. If
one assumes it is valid to differentiate under the integral and to integrate
by parts, one can prove the proposition easily; see the proof of \cite[Theorem
2.13]{Driver1995}.
\end{proof}

\begin{remark}
\label{rk.integrability} If $f$ is a matrix entry on $K$ or $K_{\mathbb{C}}$,
then by Lemma \ref{l.m.e.s.a.alg}, $|f|^{2}$ is also a matrix entry. Thus, the
absolute convergence of the integral in Proposition \ref{p.heat.op.on.m.e.}
tells us that $f$ is in $L^{2}(K,\rho_{t})$ or $L^{2}(K_{\mathbb{C}}%
,\mu_{s,\tau}).$
\end{remark}

\section{The Segal--Bargmann Transform\label{Section SBT}}

We analyze the complex-time Segal--Bargmann transform for a connected Lie
group of compact type in two stages. In the first stage, we consider a
transform $M_{\tau}$ (see Definition \ref{Mtau.def} below) defined on matrix
entries using a power-series definition of the heat operator. Using the
strategy outlined in Section \ref{proofsketch.sec} along with density results
in Theorem \ref{t.HL^2.density}, we show that $M_{\tau}$ maps a dense subspace
of $L^{2}(K,\rho_{s})$ isometrically onto a dense subspace of $\mathcal{H}%
L^{2}(K_{\mathbb{C}},\mu_{s,\tau})$. Thus, $M_{\tau}$ extends to a unitary map
$\overline{M}_{s,\tau}$ of $L^{2}(K,\rho_{s})$ onto $\mathcal{H}%
L^{2}(K_{\mathbb{C}},\mu_{s,\tau})$.

In the second stage, we show that the heat kernel $\rho_{t}(x)$ on $K$ has a
holomorphic extension in both $t$ and $x$, denoted $\rho_{\mathbb{C}}%
(\cdot,\cdot)$. We then prove that the unitary map $\overline{M}_{s,\tau}$ may
be computed by \textquotedblleft convolution\textquotedblright\ with the
holomorphically extended heat kernel. That is to say,
\[
(\overline{M}_{s,\tau}f)(z)=\int_{K}\rho_{\mathbb{C}}(\tau,zk^{-1})f(k)\,dk
\]
for all $s>0$, $f\in L^{2}(K,\rho_{s})$, $\tau\in\mathbb{D}(s,s)$, and $z\in
K_{\mathbb{C}}$.

The advantage of the two-stage approach to the proof is that we can use the
unitary map $\overline{M}_{s,\tau}$ to establish the existence of the
holomorphic extension of the heat kernel, thus avoiding the
representation-theoretic estimates used in \cite{Hall1994}. Although this
approach was used already in \cite{Driver1995}, a number of details are
different in the complex case. We therefore provide full proofs here.

\subsection{Constructing a Unitary Map\label{unitaryMap.sec}}

As usual, we work on a connected Lie group $K$ of compact type, with a fixed
$\mathrm{Ad}(K)$-invariant inner product on its Lie algebra $\mathfrak{k}$.
According to Theorem \ref{t.HL^2.density}, the space of matrix entries is
dense in $L^{2}(K,\rho_{s})$ and the space of holomorphic matrix entries is
dense in $\mathcal{H}L^{2}(K_{\mathbb{C}},\mu_{s,\tau})$.

We now define a transform $M_{\tau}$ directly by its action on matrix entries.
Let $f_{\pi,A}$ be a matrix entry on $K$ acting on a complex vector space
$V_{\pi}$. By the universal property of complexifications, the representation
$\pi$ extends uniquely to a holomorphic representation $\pi_{\mathbb{C}}$ of
$K_{\mathbb{C}}$ on $V_{\pi}$. Hence, the matrix entry $f_{\pi,A}$ has an
analytic continuation as well,
\[
(f_{\pi,A})_{\mathbb{C}}(g)=\mathrm{Tr}(\pi_{\mathbb{C}}(g)A)=f_{\pi
_{\mathbb{C}},A}(g),\qquad g\in K_{\mathbb{C}}.
\]

\begin{definition}
\label{Mtau.def}For $\tau\in\mathbb{C}_{+}$, define $M_{\tau}$ on matrix
entries on $K$ as
\[
M_{\tau}f_{\pi,A}=\left[  \sum_{n=0}^{\infty}\frac{(\tau/2)^{n}}{n!}%
(\Delta_{K})^{n}f_{\pi,A}\right]  _{\mathbb{C}}.
\]

\end{definition}

Note that, by (\ref{e.heat.op.power.series.1}), $M_{\tau}f_{\pi,A}$ is again a
matrix entry, and thus has a holomorphic extension.

\subsubsection{Complex Vector Fields and Commutation Relations}

We would now like to emulate the proof of the Segal--Bargmann isometry for the
$\mathbb{R}^{d}$ case outlined in Section \ref{proofsketch.sec}. To that end,
we must introduce the complex vector fields generalizing the complex
derivatives $\partial/\partial z_{j}$ and $\partial/\partial\bar{z}_{j}$ in
the Euclidean context.

\begin{definition}
\label{d.del.delbar} Let $G$ be a complex Lie group with Lie algebra
$\mathfrak{g}$ and let $X$ be an element of $\mathfrak{g}$. The
\textbf{holomorphic} and \textbf{antiholomorphic vector fields} associated to
$X$ are complex vector fields $\partial_{X}$ and $\bar{\partial}_{X}$ on $G$
defined by
\begin{equation}
\partial_{X}\equiv\frac{1}{2}\left(  \tilde{X}-i~\widetilde{JX}\right)
\quad\text{and}\quad\bar{\partial}_{X}\equiv\frac{1}{2}\left(  \tilde
{X}+i~\widetilde{JX}\right)  . \label{e.def.del_V}%
\end{equation}

\end{definition}

\noindent In the special case $G=\mathbb{C}^{d}$, if $X=\partial/\partial
x_{j}$ then $\partial_{X}=\partial/\partial z_{j}$ and $\bar{\partial}%
_{X}=\partial/\partial\bar{z}_{j}$. By (\ref{Jholomorphic}), if $X\in
\mathfrak{g}$ and $F$ is holomorphic on $G$ then%
\begin{align}
\partial_{X}F  &  =\tilde{X}F,\qquad\bar{\partial}_{X}%
F=0\label{DelHolomorphic}\\
\partial_{X}\bar{F}  &  =0,\quad\qquad\,\bar{\partial}_{X}F=\tilde{X}F.
\label{DelAnti}%
\end{align}

\begin{lemma}
\label{l.intermediate.commutator} If $X,V\in\mathfrak{g}$, then
\[
\lbrack\partial_{V},\widetilde{JX}]=i[\partial_{V},\tilde{X}],\qquad
\text{and}\qquad\lbrack\bar{\partial}_{V},\widetilde{JX}]=-i[\bar{\partial
}_{V},\tilde{X}].
\]

\end{lemma}

\begin{proof}
By \eqref{Jbracket}, for any $W_{1},W_{2}\in\mathfrak{g}$, $[JW_{1}%
,W_{2}]=J[W_{1},W_{2}]=[W_{1},JW_{2}]$ and therefore by the definition of the
Lie bracket,%
\[
\lbrack\widetilde{JW_{1}},\tilde{W}_{2}]=\widetilde{[JW_{1},W_{2}%
]}=\widetilde{J[W_{1},W_{2}]}=\widetilde{[W_{1},JW_{2}]}=[\tilde{W}%
_{1},\widetilde{JW_{2}}].
\]
We can then compute from the definition that
\[
\lbrack\partial_{V},\widetilde{JX}]=\frac{1}{2}[\tilde{V}-i\widetilde{JV}%
,\widetilde{JX}]=\frac{1}{2}[\widetilde{JV}-i\widetilde{JJV},\tilde{X}%
]=\frac{1}{2}[\widetilde{JV}+i\tilde{V},\tilde{X}]=i[\partial_{V},\tilde{X}].
\]
The calculation for $\bar{\partial}_{V}$ is similar.
\end{proof}

We now specialize to the case $G=K_{\mathbb{C}}$ for a compact-type Lie group
$K$.

\begin{definition}
Fix an orthonormal basis $\{X_{1},\ldots,X_{d}\}$ for $\mathfrak{k}$, and let
$\partial_{j} := \partial_{X_{j}}$ as in \eqref{e.def.del_V}. Then set
\begin{equation}
\label{e.del^2}\partial^{2} \equiv\sum_{j=1}^{d} \partial_{j}^{2},
\qquad\text{and} \qquad\bar\partial^{2} \equiv\sum_{j=1}^{d} \bar\partial
_{j}^{2}.
\end{equation}

\end{definition}

A routine calculation shows that the operators $\partial^{2}$ and
$\bar{\partial}^{2}$ are well-defined, independent of the choice of
orthonormal basis.

\begin{lemma}
\label{l.del^2.commutes.Rkhat} The operators $\partial^{2}$ and $\bar
{\partial}^{2}$ commute with the right action of $K$ on $K_{\mathbb{C}}.$
\end{lemma}

\noindent This is a routine computation and is left to the reader.

This brings us to the main commutator result of this section.

\begin{proposition}
\label{p.del^2.A.commute} For any $A\in\mathfrak{k}_{\mathbb{C}}$,
\[
[\partial^{2},\widetilde{A}]=[\bar\partial^{2},\widetilde{A}]=0.
\]

\end{proposition}

\begin{proof}
As any $A\in\mathfrak{k}_{\mathbb{C}}$ has the form $A=V+JW$ for some
$V,W\in\mathfrak{k}$, it suffices by linearity to prove that $\partial^{2}$
and $\bar{\partial}^{2}$ commute with $\widetilde{V}$ and $\widetilde{JV}$ for
any $V\in\mathfrak{k}$ . For the former statement, apply Lemma
\ref{l.del^2.commutes.Rkhat} to the right action of $k=e^{tV}$, and
differentiate at $t=0$ to yield the result. For the second statement, we
employ Lemma \ref{l.intermediate.commutator} and compute as follows.
\begin{align*}
\lbrack\partial^{2},\widetilde{JV}]=\sum_{j=1}^{d}[\partial_{j}\partial
_{j},\widetilde{JV}]  &  =\sum_{j=1}^{d}\left(  \partial_{j}[\partial
_{j},\widetilde{JV}]+[\partial_{j},\widetilde{JV}]\partial_{j}\right) \\
&  =i\sum_{j=1}^{d}\left(  \partial_{j}[\partial_{j},\tilde{V}]+[\partial
_{j},\tilde{V}]\partial_{j}\right)  =i\sum_{j=1}^{d}[\partial_{j}\partial
_{j},\tilde{V}]=i[\partial^{2},\tilde{V}]
\end{align*}
and we already showed that $[\partial^{2},\tilde{V}]=0$. A similar calculation
proves the result for $\bar{\partial}^{2}$.
\end{proof}

\begin{corollary}
\label{c.del^2.delbar^2.Laps.commute} The operators $\partial^{2}$,
$\bar{\partial}^{2}$, $\Delta_{K}$, and $\Delta_{s,\tau}$ all mutually commute.
\end{corollary}

Here we regard $\Delta_{K}$ as a left-invariant operator on $K_{\mathbb{C}}.$

\begin{proof}
Since $\Delta_{K}$ and $\Delta_{s,\tau}$ are linear combinations of squares of
left-invariant vector fields on $K_{\mathbb{C}}$, Proposition
\ref{p.del^2.A.commute} shows that they both commute with $\partial^{2}$ and
$\bar{\partial}^{2}$. Similarly, letting $Y_{j}=JX_{j}$, since $\partial
_{j}^{2}$ and $\bar{\partial}_{j}^{2}$ are linear combinations of $\tilde
{X}_{j}^{2}$, $\tilde{Y}_{j}^{2}$, and $\tilde{X}_{j}\tilde{Y}_{j}=\tilde
{Y}_{j}\tilde{X}_{j}$ (cf. \eqref{Jbracket}), the commutator $[\partial
^{2},\bar{\partial}^{2}]=0$ also follows from Proposition
\ref{p.del^2.A.commute}. Now, since $\Delta_{s,\tau}$ is the Laplacian
associated to an $\mathrm{Ad}(K)$-invariant inner product on $\mathfrak{k}%
_{\mathbb{C}}$, it commutes with the right action of $K.$ Thus, $\Delta
_{s,\tau}$ commutes with each $\tilde{X}_{j}$ and thus with $\Delta_{K}.$
\end{proof}

\begin{remark}
The fact that $[\partial^{2},\bar{\partial}^{2}]=0$ holds quite generally.
Indeed, on any complex manifold, if $Z=\sum_{j}a_{j}(z)\frac{\partial
}{\partial z_{j}}$ and $W=\sum_{j}b_{j}(z)\frac{\partial}{\partial z_{j}}$ are
two holomorphic vector fields, than a simple computation shows that
$[Z,\overline{W}]=0$.
\end{remark}

\subsubsection{The Transform $M_{\tau}$, and the Isomorphism $\overline
{M}_{s,\tau}$}

The usefulness of the $\partial^{2}$ and $\bar{\partial}^{2}$ operators and
the commutation result in Corollary \ref{c.del^2.delbar^2.Laps.commute} in the
present context lies in the following result.

\begin{lemma}
\label{deltaKsum.lem}Let $s>0$ and $\tau\in\mathbb{D}(s,s)$. Let
$\Delta_{s,\tau}$ denote the $K_{\mathbb{C}}$ Laplacian of Definition
\ref{d.Lap.rsu}, and let $\Delta_{K}$ denote the Laplacian of $K$ acting on
$C^{\infty}(K_{\mathbb{C}})$ as usual. Then
\[
s\Delta_{K}=\Delta_{s,\tau}+\tau\partial^{2}+\bar{\tau}\bar{\partial}^{2}%
\]
where all operators appearing in this identity are mutually commuting.
\end{lemma}

\begin{proof}
Fix an orthonormal basis $\{X_{1},\ldots,X_{d}\}$ of $\mathfrak{k}$. For ease
of reading, let $Y_{j}=JX_{j}$. To begin, we compute that, for each $j$,
\begin{align}
\partial_{j}^{2}+\bar{\partial}_{j}^{2}  &  =\frac{1}{4}(\tilde{X}_{j}%
-i\tilde{Y}_{j})^{2}+\frac{1}{4}(\tilde{X}_{j}+i\widetilde{Y}_{j})^{2}%
=\frac{1}{2}(\tilde{X}_{j}^{2}-\tilde{Y}_{j}^{2}),\label{e.del^2-delbar^2}\\
\partial_{j}^{2}-\bar{\partial}_{j}^{2}  &  =\frac{1}{4}(\tilde{X}_{j}%
-i\tilde{Y}_{j})^{2}-\frac{1}{4}(\tilde{X}_{j}+i\tilde{Y}_{j})^{2}=-i\tilde
{X}_{j}\tilde{Y}_{j}%
\end{align}
where we have used the fact that $[\tilde{X}_{j},\tilde{Y}_{j}]=0$ (cf.\ \eqref{Jbracket}).

Now, let $\tau=t+iu$. Then for each $j$,
\[
\tau\partial_{j}^{2}+\bar{\tau}\bar{\partial}_{j}^{2}=t(\partial_{j}^{2}%
+\bar{\partial}_{j}^{2})+iu(\partial_{j}^{2}-\bar{\partial}_{j}^{2})=\frac
{t}{2}(\tilde{X}_{j}^{2}-\tilde{Y}_{j}^{2})+u\tilde{X}_{j}\tilde{Y}_{j}.
\]
Thus, we have
\begin{equation}
\left[  \left(  s-\frac{t}{2}\right)  \tilde{X}_{j}^{2}+\frac{t}{2}\tilde
{Y}_{j}^{2}-u\tilde{X}_{j}\tilde{Y}_{j}\right]  +\tau\partial_{j}^{2}%
+\bar{\tau}\bar{\partial}_{j}^{2}=s\tilde{X}_{j}^{2}. \label{e.tau.del^2.K}%
\end{equation}
Summing \eqref{e.tau.del^2.K} on $j$ proves the lemma.
\end{proof}

We can now prove that $M_{\tau}$ is a bijection from the space of matrix
entries on $K$ to the space of holomorphic matrix entries on $K_{\mathbb{C}}$,
isometric from $L^{2}(K,\rho_{s})$ into $L^{2}(K_{\mathbb{C}},\mu_{s,\tau})$.

\begin{theorem}
\label{t.isometry.K} Let $f$ be a matrix entry function on $K$. Then for $s>0$
and $\tau\in\mathbb{D}(s,s)$,
\begin{equation}
\Vert M_{\tau}f\Vert_{L^{2}(K_{\mathbb{C}},\mu_{s,\tau})}=\Vert f\Vert
_{L^{2}(K,\rho_{s})}. \label{isomMatrixEntries}%
\end{equation}
Moreover, every holomorphic matrix entry $F$ on $K_{\mathbb{C}}$ has the form
$F=M_{\tau}f$ for some matrix entry $f$ on $K$.
\end{theorem}

The proof of (\ref{isomMatrixEntries}) follows the strategy outlined in
Section \ref{proofsketch.sec}, using left-invariant vector fields in place of
the partial derivatives in the Euclidean case. A key step in the argument
requires us to combine exponentials, which is possible only if the operators
in the exponent commute. It is at this point that we use the commutativity
result in Corollary \ref{c.del^2.delbar^2.Laps.commute}.

\begin{proof}
Let $F=M_{\tau}f$. The matrix entry $f$ on $K$ has a holomorphic extension
$f_{\mathbb{C}}$ to $K_{\mathbb{C}}$. Now, $\Delta_{K},$ viewed as a
left-invariant differential operator on $K_{\mathbb{C}}$, is a sum of squares
of left-invariant vector fields; thus, it preserves the space of holomorphic
functions. Thus, we have that $((\Delta_{K})^{n}f)_{\mathbb{C}}=(\Delta
_{K})^{n}(f_{\mathbb{C}})$ for all $n\geq0$. It follows that $F$ may be
computed as $F=e^{\tau\Delta_{K}/2}(f_{\mathbb{C}})$. Since $f_{\mathbb{C}}$
is holomorphic, we may use (\ref{DelHolomorphic}) to rewrite this relation as%
\[
F=e^{\tau\partial^{2}/2}(f_{\mathbb{C}}).
\]
It is then straightforward, using (\ref{DelHolomorphic}) and (\ref{DelAnti}),
to see that%
\[
\left\vert F\right\vert ^{2}=e^{\tau\partial^{2}/2}e^{\bar{\tau}\bar{\partial
}^{2}/2}(f_{\mathbb{C}}\bar{f}_{\mathbb{C}}).
\]
Thus, using Proposition \ref{p.heat.op.on.m.e.}, we may compute the norm of
$F$ as
\begin{align}
\Vert F\Vert_{L^{2}(K_{\mathbb{C}},\mu_{s,\tau})}^{2}  &  =\left(
e^{\Delta_{s,\tau}/2}|F|^{2}\right)  (e)\nonumber\\
&  =\left(  e^{\Delta_{s,\tau}/2}e^{\tau\partial^{2}/2}e^{\bar{\tau}%
\bar{\partial}^{2}/2}(f_{\mathbb{C}}\bar{f}_{\mathbb{C}})\right)  (e).
\label{normFcalc}%
\end{align}

By the commutativity result in Corollary \ref{c.del^2.delbar^2.Laps.commute},
we may combine the exponents in the last expression in (\ref{normFcalc}). Note
that there are no domain issues to worry about here: All the exponentials in
(\ref{normFcalc}) are defined by power series and since $f_{\mathbb{C}}\bar
{f}_{\mathbb{C}}$ is a matrix entry (cf.\ Lemma \ref{l.m.e.s.a.alg}), all
exponentials are acting in a fixed finite-dimensional subspace of functions on
$K_{\mathbb{C}}$. Using Lemma \ref{deltaKsum.lem}, (\ref{normFcalc}) therefore
becomes%
\[
\Vert F\Vert_{L^{2}(K_{\mathbb{C}},\mu_{s,\tau})}^{2}=(e^{s\Delta_{K}%
/2}\left\vert f_{\mathbb{C}}\right\vert ^{2})(e)=(e^{s\Delta_{K}/2}\left\vert
f\right\vert ^{2})(e).
\]
The last equality holds because $e$ belongs to $K$ and $\Delta_{K}$ is a sum
of squares of left-invariant vector fields associated to elements of
$\mathfrak{k}$. Using Proposition \ref{p.heat.op.on.m.e.} again, we finally
conclude that%
\[
\Vert F\Vert_{L^{2}(K_{\mathbb{C}},\mu_{s,\tau})}^{2}=\left\Vert f\right\Vert
_{L^{2}(K,\rho_{s})}^{2}%
\]
establishing (\ref{isomMatrixEntries}).

Suppose now that $F$ is a holomorphic matrix entry on $K_{\mathbb{C}}$; that
is, $F=f_{\pi_{\mathbb{C}},A}$ for some finite-dimensional holomorphic
representation $\pi_{\mathbb{C}}$ of $K_{\mathbb{C}}$. Then $\left.
F\right\vert _{K}=f_{\pi,A}$, where $\pi$ is the restriction of $\pi
_{\mathbb{C}}$ to $K$. We may then define%
\[
f=e^{-\frac{\tau}{2}\Delta_{K}}(\left.  F\right\vert _{K})=f_{\pi
,e^{-\frac{\tau}{2}C_{\pi}}A}.
\]
Then $f$ is a matrix entry and we have $M_{\tau}f=(e^{\frac{\tau}{2}\Delta
_{K}}f)_{\mathbb{C}}=F$.
\end{proof}

\begin{theorem}
\label{MstUnitary.thm}The map $M_{\tau}$ has a unique continuous extension to
$L^{2}(K,\rho_{s})$, denoted $\overline{M}_{s,\tau}$, and this extension is a
unitary map from $L^{2}(K,\rho_{s})$ onto $\mathcal{H}L^{2}(K_{\mathbb{C}}%
,\mu_{s,\tau})$.
\end{theorem}

\begin{proof}
Theorem \ref{t.HL^2.density} tells us that $M_{\tau}$ is defined on a dense
subspace of $L^{2}(K,\rho_{s})$. Since $M_{\tau}$ is isometric, the bounded
linear transformation theorem (e.g., Theorem I.7 in \cite{Reed1980}) tells us
that $M_{\tau}$ has a unique continuous extension to a map $\overline
{M}_{s,\tau}$ of $L^{2}(K,\rho_{s})$ into $\mathcal{H}L^{2}(K_{\mathbb{C}}%
,\mu_{s,\tau})$. This extension is easily seen to be isometric, and since (by
Theorem \ref{t.HL^2.density} again) the image of $M_{\tau}$ is dense, the
extension is actually a unitary map.
\end{proof}

For a general $f\in L^{2}(K,\rho_{s})$, the value of $\overline{M}_{s,\tau}$
may be computed by approximating $f$ by a sequence $f_{n}$ of matrix entries
and setting%
\begin{equation}
\overline{M}_{s,\tau}f=\lim_{n\rightarrow\infty}M_{\tau}f_{n}.
\label{mstBarDef}%
\end{equation}
(The bounded linear transformation theorem guarantees that the limit exists
and that the value of $\overline{M}_{s,\tau}$ is independent of the choice of
approximating sequence.) Now, (\ref{mstBarDef}) is not a very convenient way
to compute. In the next section, we will seek a direct way of computing
$\overline{M}_{s,\tau}$, which will also demonstrate that $\overline
{M}_{s,\tau}$ coincides with the way we defined the complex-time
Segal--Bargmann transform in the introduction; cf.\ \eqref{e.SB.Ctime}. A
first step in that direction is proving that $(\overline{M}_{s,\tau}f)(z)$ is
holomorphic in both $\tau$ and $z$.

\begin{lemma}
\label{mstauHolo.lem}Fix $s>0$. For each $f\in L^{2}(K,\rho_{s})$, the
function $(\tau,z)\mapsto(\overline{M}_{s,\tau}f)(z)$ is a holomorphic
function on $\mathbb{D}(s,s)\times K_{\mathbb{C}}$.
\end{lemma}

\begin{proof}
If $f=f_{\pi,A}$ is a matrix entry, then
\[
(\overline{M}_{s,\tau}f_{\pi,A})(z)=(M_{\tau}f_{\pi,A})(z)=\mathrm{Tr}%
(\pi_{\mathbb{C}}(g)e^{\tau C_{\pi}/2}A)
\]
which is easily seen to depend holomorphically on $\tau$ and $z$.

We then approximate an arbitrary $f\in L^{2}(K,\rho_{s})$ by a sequence
$f_{n}$ of matrix entries. Then $M_{\tau}f_{n}=\overline{M}_{s,\tau}f_{n}$
will converge to $\overline{M}_{s,\tau}f$ in $\mathcal{H}L^{2}(K_{\mathbb{C}%
},\mu_{s,\tau})$. It is well known that the evaluation map $F\mapsto F(z)$ on
$\mathcal{H}L^{2}(K_{\mathbb{C}},\mu_{s,\tau})$ is a bounded linear
functional; this is due to the ubiquitous pointwise $L^{2}$ estimates in this
holomorphic space (cf.\ \cite{Driver2015,Hall2000}). We claim that we can
actually find locally uniform bounds on this functional. That is to say: for
each precompact open subset $U$ of $K_{\mathbb{C}}$ and $r\in(0,s)$, there
exists $C=C(r,U)<\infty$ such that, for all $\tau\in\mathbb{D}(s,r)$ and
$F\in\mathcal{H}L^{2}(K_{\mathbb{C}},\mu_{s,\tau})$,%
\begin{equation}
\sup_{z\in U}\left\vert F(z)\right\vert \leq C(r,U) \left\Vert F\right\Vert
_{L^{2}(K_{\mathbb{C}},\mu_{s,\tau})}. \label{localBounds}%
\end{equation}
Assuming this result for the moment, we can conclude that the convergence of
$(\overline{M}_{s,\tau}f_{n})(z)$ to $(\overline{M}_{s,\tau}f)(z)$ is locally
uniform jointly in $(\tau,z)$, and since each function in the sequence is
holomorphic, it follows that the limit $(\overline{M}_{s,\tau}f)(z)$ is
jointly holomorphic in $(\tau,z)$ as claimed.

To establish the bound in \eqref{localBounds}, we observe that the norm of the
pointwise evaluation functional can be estimated in terms of lower bounds on
the density $\mu_{s,\tau}$. For example, \cite[Theorem 3.6]{Driver2015} shows
(in our context) that, for any precompact neighborhood $V$ of the identity
$e$, there is a constant $C(V)$ so that, for all holomorphic $F$ and $z\in
K_{\mathbb{C}}$,
\[
|F(z)|\leq\frac{C(V)}{\inf_{v\in V}\sqrt{\mu_{s,\tau}(vz)}}\Vert F\Vert
_{L^{2}(K_{\mathbb{C}},\mu_{s,\tau})}.
\]
The constant $C(V)$ is determined only by the holomorphic structure of the
group (given by averaging a symmetrized bump function on $V$, applying the
Cauchy integral formula); hence, $C(V)$ is independent of $s$ and $\tau$.
Hence, it suffices to show that $\mu_{s,\tau}(z)$ is bounded strictly above
$0$ locally uniformly in $\tau$ and $z$.

Since $K_{\mathbb{C}}$ factors as $(K_{0})_{\mathbb{C}}\times\mathbb{C}^{d}$
(recall Proposition \ref{p.compact.type.decomp}), the heat kernel $\mu
_{s,\tau}$ also factors over this product. On the $\mathbb{C}^{d}$ side, there
is an explicit formula for $\mu_{s,\tau}(z)$ (given in \eqref{muRk}) which is
manifestly bounded away from zero locally in both $\tau$ and $z$. Thus, it
suffices to assume that $K$ is compact, which we do from now on.

Denote $t=\mathrm{Re}\,\tau$. From the averaging lemma (Lemma
\ref{averaging.lem}) and Proposition \ref{nutIndep.prop}, we see that there is
a strictly positive constant $C^{\prime}(s,\tau)$ such that $\mu_{s,\tau
}\asymp_{C^{\prime}(s,\tau)}\mu_{t,t}$, and the constants can be chosen to
depend continuously on $(s,\tau)$. Note that $\mu_{t,t}$ is the heat kernel
for a single metric, which is therefore a continuous positive function of
$(t,z)\in(0,\infty)\times K_{\mathbb{C}}$. In particular, $\mu_{t,t}(z)$ is
bounded strictly away from $0$ for $(t,z)$ in compact subsets of
$(0,\infty)\times K_{\mathbb{C}}$. It follows from the continuity of the
function $(s,\tau)\mapsto C^{\prime}(s,\tau)$ that the same holds true for
$\mu_{s,\tau}(z)$, establishing \eqref{localBounds} and completing the proof.
\end{proof}

\subsection{The Analytic Continuation of the Heat Kernel}

In this section, we show that the unitary map $\overline{M}_{s,\tau}\colon
L^{2}(K,\rho_{s})\rightarrow\mathcal{H}L^{2}(K_{\mathbb{C}},\mu_{s,\tau})$
constructed in Section \ref{unitaryMap.sec} may be computed as a
\textquotedblleft convolution\textquotedblright\ against a holomorphic
extension of the heat kernel $\rho_{t}$ on $K$. The following theorem makes
this precise.

\begin{theorem}
\label{analyticContHeat.thm}Let $K$ be a compact-type Lie group.

\begin{enumerate}
\item \label{one.point}There exists a unique holomorphic function
$\rho_{\mathbb{C}}:\mathbb{C}_{+}\times K_{\mathbb{C}}\rightarrow\mathbb{C}$
such that for $t>0$ and $x\in K$ we have%
\[
\rho_{\mathbb{C}}(t,x)=\rho_{t}(x).
\]

\item \label{two.point}If $s>0$ and $\tau\in\mathbb{D}(s,s)$, then for each
$z\in K_{\mathbb{C}}$, the function%
\[
x\mapsto\frac{\rho_{\mathbb{C}}(\tau,zx^{-1})}{\rho_{s}(x)}%
\]
belongs to $L^{2}(K,\rho_{s})$.

\item \label{three.point}The unitary map $\overline{M}_{s,\tau}$ may be
computed as%
\[
(\overline{M}_{s,\tau}f)(z)=\int_{K}\rho_{\mathbb{C}}(\tau,zk^{-1})f(k)\,dk
\]
for all $f\in L^{2}(K,\rho_{s})$ and all $z\in K_{\mathbb{C}}$.
\end{enumerate}
\end{theorem}

Since
\[
\rho_{\mathbb{C}}(\tau,zk^{-1})f(k)\,dk=\frac{\rho_{\mathbb{C}}(\tau,zk^{-1}%
)}{\rho_{s}(k)}f(k)~\rho_{s}(k)dk
\]
it follows by the Cauchy--Schwarz inequality and Theorem
\ref{analyticContHeat.thm}\eqref{two.point} that the function $k\mapsto
\rho_{\mathbb{C}}(\tau,zk^{-1})f(k)$ is integrable. Using the decomposition of
$K$ as $K_{0}\times\mathbb{R}^{d}$, where $K_{0}$ is compact (Proposition
\ref{p.compact.type.decomp}), we may easily reduce the general case to the
compact case and the Euclidean case, which we now address separately.

\subsubsection{The Compact Case}

It is possible to construct the holomorphic extension of the heat kernel on
$K$ using the method of \cite[Section 4]{Hall1994}, which is based on a
term-by-term analytic continuation of the expansion of the heat kernel in
terms of characters. Indeed, replacing $t$ by $t+iu$ in the heat kernel makes
no change to the (absolute) convergence estimates in \cite{Hall1994}. (The
time-parameter occurs only linearly in the exponent there, so the absolute
value of each term would be independent of $u$.) On the other hand, the
argument in \cite{Hall1994} requires detailed knowledge of the representation
theory of $K$. We present here a different argument (similar to the proof of
Corollary 4.6 in \cite{Driver1995}) that uses the unitary map $\overline
{M}_{s,\tau}$ of Theorem \ref{MstUnitary.thm} to construct the desired
analytic continuation.

\begin{lemma}
\label{mstauHeat.lem}If $K$ is compact, $s>0,0<t<2s$, and $\overline{M}_{s,t}$
is the unitary map as in Theorem \ref{MstUnitary.thm}, then for any $f\in
L^{2}(K,\rho_{s})$,%
\begin{equation}
(\overline{M}_{s,t}f)\left(  x\right)  =\left(  \rho_{t}\ast f\right)  \left(
x\right)  =\int_{K}\rho_{t}(xk^{-1})f(k)\,dk\text{ }\forall~x\in K\subset
K_{\mathbb{C}}. \label{MstauHeat}%
\end{equation}
(Note: for $K$ compact, $L^{2}\left(  K\right)  =L^{2}(K,\rho_{s})$
independent of $s>0$ and hence $\overline{M}_{s,t}f$ does not really depend on
$s$.)
\end{lemma}

\begin{proof}
By Definition \ref{Mtau.def}, we have that for any matrix entry $f_{\pi,A}$ on
$K,$
\[
(\overline{M}_{s,t}f_{\pi,A})(g)=(M_{t}f_{\pi,A})(g)=\mathrm{Tr}%
(\pi_{\mathbb{C}}(g)e^{tC_{\pi}/2}A),
\]
where $\pi_{\mathbb{C}}$ is the holomorphic extension of $\pi$ from $K$ to
$K_{\mathbb{C}}.$ Thus, by (\ref{e.heat.op.power.series.1}), we have
\[
\left(  \overline{M}_{s,t}f_{\pi,A}\right)  |_{K}(x)=\mathrm{Tr}%
(\pi(x)e^{tC_{\pi}/2}A)\left(  \rho_{t}\ast f_{\pi,A}\right)  \left(
x\right)  .
\]
This suffices to complete the proof as matrix entries are dense in
$L^{2}\left(  K\right)  $ and both $L^{2}\left(  K\right)  \ni f\rightarrow
\left(  \overline{M}_{s,t}f\right)  \left(  x\right)  \in\mathbb{C}$ and
$L^{2}\left(  K\right)  \ni f\rightarrow\left(  \rho_{t}\ast f\right)  \left(
x\right)  \in\mathbb{C}$ are continuous linear functionals on $L^{2}\left(
K\right)  $ for each fixed $x\in K.$ The first assertion holds since
$\overline{M}_{s,\tau}:L^{2}(K,\rho_{s})\rightarrow\mathcal{H}L^{2}%
(K_{\mathbb{C}},\mu_{s,\tau})$ is unitary and pointwise evaluation on
$\mathcal{H}L^{2}(K_{\mathbb{C}},\mu_{s,\tau})$ is continuous and the second
follows by H\"{o}lder's inequality.
\end{proof}

\begin{proof}
[Proof of Theorem \ref{analyticContHeat.thm} in the compact group case]We
begin with point \eqref{one.point}: the space-time analytic continuation of
the heat kernel. Let $0<\delta<r<\infty$, and consider the vertically
symmetric rectangle $U_{\delta,r}=\{\tau\in\mathbb{C}_{+}\colon\delta
<\mathrm{Re}\,\tau<r,\;|\mathrm{Im}\,\tau|<r\}$. Let $0<\epsilon<\delta$, and
fix $s>0$ large enough that $U_{\delta,r}-\epsilon\subset\mathbb{D}(s,s)$. The
function $\rho_{\epsilon}$ is continuous and hence in $L^{2}(K,\rho_{s})$. We
then define $\rho_{\mathbb{C}}\colon U_{\delta,r}\times K_{\mathbb{C}%
}\rightarrow\mathbb{C}$ by
\begin{equation}
\rho_{\mathbb{C}}(\tau,z)=\left(  \overline{M}_{s,\tau-\epsilon}\rho
_{\epsilon}\right)  (z). \label{d.rho.C.compact}%
\end{equation}
By Lemma \ref{mstauHolo.lem}, $\rho_{\mathbb{C}}$ is analytic in both
variables so long as $\tau-\epsilon\in\mathbb{D}(s,s)$; in particular,
$\rho_{\mathbb{C}}$ is analytic on $U_{\delta,r}\times K_{\mathbb{C}}$. For
the moment, it appears a priori that the value of $\rho_{\mathbb{C}}$ depends
on $s$ and $\epsilon$.

Now consider the restriction of $\rho_{\mathbb{C}}$ to $(t,x)\in(U_{\delta
,r}\cap\mathbb{R})\times K$. By lemma \ref{mstauHeat.lem} and the semigroup
property of the heat kernel,
\begin{equation}
\rho_{\mathbb{C}}(t,x)=\left(  \overline{M}_{s,t-\epsilon}\rho_{\epsilon
}\right)  (x)=\left(  \rho_{t-\epsilon}\ast\rho_{\epsilon}\right)  \left(
x\right)  =\rho_{t}(x)\text{ }\forall~x\in K. \label{equ.rtt}%
\end{equation}
Thus, $\rho_{\mathbb{C}}$ is a holomorphic extension of the heat kernel
$\rho_{t}(x)$ in $t$ and $x$. Analytic continuation from $K$ to $K_{\mathbb{C}%
}$ is unique (cf.\ \cite[Lemma 4.11.13]{Varadarajan}), and also from
$U_{\delta,r}\cap\mathbb{R}$ to $U_{\delta,r}$ by elementary complex analysis.
In particular, since $\rho_{t}(x)$ does not depend on $s$ or $\epsilon$,
neither does the function $\rho_{\mathbb{C}}$.

Thus, for each rectangle $U_{\delta,r}$, there is a unique analytic
continuation of the heat kernel to a holomorphic function $\rho_{\mathbb{C}}$
on $U_{\delta,r}\times K_{\mathbb{C}}$. Let $\delta_{n}$ and $r_{n}$ be
sequences with $\delta_{n}\downarrow0$ and $r_{n}\uparrow\infty$, let
$U_{n}=U_{\delta_{n},r_{n}}$, and let $\rho_{\mathbb{C}}^{n}$ be the analytic
continuation of $\rho_{t}(x)$ to $U_{n}$. The rectangles $U_{n}$ are nested
with union $\mathbb{C}_{+}$; since $\rho_{\mathbb{C}}^{n}$ and $\rho
_{\mathbb{C}}^{m}$ agree on $\left(  U_{n\wedge m}\cap\mathbb{R}\right)
\times K$, uniqueness of analytic continuation shows that they agree on their
common domain $U_{n\wedge m}\times K_{\mathbb{C}}$. Thus, there is a globally
defined holomorphic function $\rho_{\mathbb{C}}$ whose value in $U_{n}\times
K_{\mathbb{C}}$ is $\rho_{\mathbb{C}}^{n}$, and thus restricts to $\rho
_{t}(x)$ on $(U_{n}\cap\mathbb{R})\times K$; ergo $\rho_{\mathbb{C}}%
(t,x)=\rho_{t}(x)$ for $t>0$ and $x\in K$, as desired. Uniqueness again
follows from \cite[Lemma 4.11.13]{Varadarajan}. This establishes point \eqref{one.point}.

Point \eqref{two.point} is immediate since $K$ is compact and the function in
question is continuous. For point \eqref{three.point}, we first note that, by
Lemma \ref{mstauHolo.lem}, $(\overline{M}_{s,\tau}f)(z)$ is holomorphic in
$\tau$ and $z$. Meanwhile, since $\rho_{\mathbb{C}}(\tau,zk^{-1})$ is
holomorphic in $\tau$ and $z$ for each fixed $k\in K$, we may use Fubini's
theorem and Morera's theorem to verify that $\int_{K}\rho_{\mathbb{C}}%
(\tau,zk^{-1})f(k)\,dx$ is also holomorphic in $\tau$ and $z$. Since both
sides of the desired equality are holomorphic in $\tau$ and $z$, it suffices
by uniqueness of analytic continuation to verify the result when $\tau
=t\in\left(  0,2s\right)  $ and $z=x$ belongs to $K$. Using Lemma
\ref{mstauHeat.lem} and the defining property of $\rho_{\mathbb{C}}$, the
desired equality thus becomes%
\[
(e^{t\Delta_{K}/2}f)(x)=\int_{K}\rho_{t}(xk^{-1})f(k)\,dk,
\]
which is true. This concludes the proof.
\end{proof}

\subsubsection{The Euclidean Case\label{RkCase.sec}}

The heat kernel $\rho_{s}$ on $\mathbb{R}^{d}$ is explicitly known to be the
Gaussian density mentioned in the introduction:
\[
\rho_{s}(x)=(2\pi s)^{-d/2}\exp\left(  -\frac{|x|^{2}}{2s}\right)
\]
and the density $\mu_{s,\tau}(z)$ in this case has been described in
\eqref{muRk} in the introduction.

\begin{proof}
[Proof of Theorem \ref{analyticContHeat.thm} in the Euclidean case]For point
\eqref{one.point}, the desired holomorphic extension is given by
\begin{equation}
\rho_{\mathbb{C}}(\tau,z):=\left(  \sqrt{2\pi\tau}\right)  ^{-d}\exp\left(
-\frac{z\cdot z}{2\tau}\right)  \label{e.rhoC.Euclidean}%
\end{equation}
where $z\cdot z=\sum_{j=1}^{d}z_{j}^{2}$ and where $\sqrt{2\pi\tau}$ is
defined by the standard branch of the square root\ (with branch cut along the
negative real axis).

Point \eqref{two.point} of the theorem is an elementary computation. Using
additive notation for the group operation, we need to verify that%
\begin{equation}
\int_{\mathbb{R}^{d}}\frac{\left\vert \rho_{\mathbb{C}}(\tau,z-x)\right\vert
^{2}}{\rho_{s}(x)^{2}}\rho_{s}(x)\,dx<\infty\label{rhoNorm}%
\end{equation}
for all $z\in\mathbb{C}^{d}$, provided that $s>0$ and $\tau\in\mathbb{D}(s,s)$
(or, equivalently, provided that $\alpha>0$; cf.\ \eqref{e.alpha}). Equation
\eqref{rhoNorm} is a Gaussian integral whose computation is tedious but
straightforward. (The integral factors into separate integrals over each copy
of $\mathbb{R}$, which may then be evaluated in a computer algebra system.) We
record the result here: if $z=\xi+i\eta$ and $\tau=t+iu$, then%
\begin{equation}
\int_{\mathbb{R}^{d}}\frac{\left\vert \rho_{\mathbb{C}}(\tau,z-x)\right\vert
^{2}}{\rho_{s}(x)^{2}}\rho_{s}(x)\,dx=\left(  \frac{\pi s}{\sqrt{\alpha}%
}\right)  ^{d}\exp\left(  \frac{t/2}{2\alpha}|\xi|^{2}+\frac{s-t/2}{2\alpha
}|\eta|^{2}+\frac{u}{2\alpha}\xi\cdot\eta\right)  \label{rhoCRk}%
\end{equation}
where, as in (\ref{e.alpha}), $\alpha=(2st-t^{2}-u^{2})/4$.

For point \eqref{three.point}, we must show that $(\overline{M}_{s,\tau}f)(z)$
may be computed as%
\begin{equation}
(\overline{M}_{s,\tau}f)(z)=\int_{\mathbb{R}^{d}}\rho_{\mathbb{C}}%
(\tau,z-x)f(x)\,dx \label{mstauPoly}%
\end{equation}
for all $f\in L^{2}(\mathbb{R}^{d},\rho_{s})$. If $f$ is a polynomial (and
thus a matrix entry) and $\tau\in\mathbb{R}$ and $z\in\mathbb{R}^{d}$,
(\ref{mstauPoly}) follows from Proposition \ref{p.heat.op.on.m.e.}.
Furthermore, when $f$ is a polynomial, both sides of (\ref{mstauPoly}) are
holomorphic in $\tau$ and $z$, so the result continues to hold when $\tau
\in\mathbb{C}_{+}$ and $z\in\mathbb{C}^{d}$. Now, both sides of
(\ref{mstauPoly}) depend continuously on $f\in L^{2}(\mathbb{R}^{d},\rho_{s}%
)$---the left-hand side by the unitarity of $\overline{M}_{s,\tau}$ and the
continuity of pointwise evaluation, and the right-hand side by the fact that
$\rho_{\mathbb{C}}(t,z-x)$ is square-integrable in $x$. Thus, we may pass to
the limit starting from polynomials to obtain the result for all $f\in
L^{2}(\mathbb{R}^{d},\rho_{s})$, thus completing the proof of Theorem
\ref{analyticContHeat.thm} in the $\mathbb{R}^{d}$ case.
\end{proof}

We note that, by (\ref{rhoCRk}), we have bounds on the value of $(\overline
{M}_{s,\tau}f)(z)$ in terms of the $L^{2}$ norm of $f$. Since $\overline
{M}_{s,\tau}$ maps isometrically onto $\mathcal{H}L^{2}(\mathbb{C}^{d}%
,\mu_{s,\tau})$, these bounds translate into pointwise bounds in
$\mathcal{H}L^{2}(\mathbb{C}^{d},\mu_{s,\tau})$ as follows:%
\begin{equation}
\left\vert F(\xi+i\eta)\right\vert ^{2}\leq\left(  \frac{\pi s}{\sqrt{\alpha}%
}\right)  ^{d}\exp\left(  \frac{t/2}{2\alpha}|\xi|^{2}+\frac{s-t/2}{2\alpha
}|\eta|^{2}+\frac{u}{2\alpha}\xi\cdot\eta\right)  \left\Vert F\right\Vert
_{L^{2}(\mathbb{C}^{d},\mu_{s,\tau})}^{2}, \label{boundsCk}%
\end{equation}
where $\mu_{s,\tau}$ is given as in \eqref{muRk}. Note that the bounds on
$\left\vert F(z)\right\vert ^{2}$ are, up to a constant, just the reciprocal
of the density $\mu_{s,\tau}$. This is typical behavior for $\mathcal{H}L^{2}$
spaces over $\mathbb{C}^{d}$ with respect to a Gaussian measure.

\subsection{The $s\to\infty$ Limit\label{section s-->infinity}}

Throughout this section, we assume that the compact-type group $K$ is actually
compact and we normalize the Haar measure $dk$ on $K$ to be a probability
measure. Recall that $\nu_{t}\in C^{\infty}\left(  K_{\mathbb{C}}%
,(0,\infty)\right)  $ is the $K$-averaged heat kernel measure, as in
Definition \ref{d.K-av.heat.kernel}.

\begin{proposition}
\label{nutIndep.prop}For all $s>0$ and $\tau=t+iu$ with $\tau\in
\mathbb{D}(s,s),$ we have%
\begin{equation}
\int_{K}\mu_{s,\tau}(gk)~dk=\nu_{t}(g). \label{nutInd}%
\end{equation}
That is to say, the integral on the left-hand side of (\ref{nutInd}) is
independent of $s$ and $u$ and therefore equals its value when $u=0$ and
$s=t,$ which is $\nu_{t}.$
\end{proposition}

For the moment, we give only a heuristic argument for Proposition
\ref{nutIndep.prop}; a full proof requires some functional-analytic
technicalities, which will be provided in Appendix \ref{technicalities.app}.
By Corollary \ref{c.del^2.delbar^2.Laps.commute}, the three terms in the
definition (\ref{e.def.Delta.s.tau}) of $\Delta_{s,\tau}$ all commute with one
another. Thus, \textit{formally}, we can differentiate in the naive way, as if
the terms in the exponents were scalars rather than operators. Assuming this
approach is valid, we would get%
\begin{equation}
\frac{\partial\mu_{s,\tau}}{\partial s}=\sum_{j=1}^{\dim\mathfrak{k}}\tilde
{X}_{j}^{2}\mu_{s,\tau};\quad\frac{\partial\mu_{s,\tau}}{\partial u}%
=\sum_{j=1}^{\dim\mathfrak{k}}\tilde{X}_{j}\tilde{Y}_{j}\mu_{s,\tau}.
\label{diffMustua}%
\end{equation}
We now denote the integral on the left-hand side of (\ref{nutInd}) by
$\nu_{s,\tau}.$ Then (\ref{diffMustua}) would tell us that
\[
\frac{\partial\nu_{s,\tau}}{\partial s}=\sum_{j=1}^{\dim\mathfrak{k}}\tilde
{X}_{j}^{2}\nu_{s,\tau};\quad\frac{\partial\nu_{s,\tau}}{\partial u}%
=\sum_{j=1}^{\dim\mathfrak{k}}\tilde{X}_{j}\tilde{Y}_{j}\nu_{s,\tau}.
\]
But $\nu_{s,\tau}$ is by construction invariant under the right action of $K,$
so that $\tilde{X}_{j}\nu_{s,\tau}=0.$ Since $\tilde{X}_{j}$ commutes with
$\tilde{Y}_{j}=\widetilde{JX}_{j},$ we would find that $\nu_{s,\tau}$ is
independent of $s$ and $u,$ as claimed.

We will use the following well-known result for the heat kernel measure on a
compact Lie group at large time.

\begin{lemma}
\label{compactRhoSinfinity.lem}If $K$ is a compact Lie group, the heat kernel
$\rho_{s}$ converges to the constant 1 uniformly over $K$ (exponentially fast)
as $s\to\infty$.
\end{lemma}

This result holds more generally on compact Riemannian manifolds. In the case
of a compact Lie group, the result follows easily from the expansion of the
heat kernel in terms of characters.

With these results in hand, we may now prove Theorem \ref{largeS.thm},
describing the large-$s$ limit of the transform $B_{s,\tau}$.

\begin{proof}
[Proof of Theorem \ref{largeS.thm}]Since $K$ is compact, the function
$\rho_{s}$ is bounded and bounded away from zero, showing that $L^{2}%
(K)=L^{2}(K,\rho_{s})$ as sets. The equality of $L^{2}(K_{\mathbb{C}},\nu
_{t})$ and $L^{2}(K_{\mathbb{C}},\mu_{s,\tau})$ as sets follows from the
averaging lemma (Lemma \ref{averaging.lem}) and Proposition
\ref{nutIndep.prop}. We then note that as $s$ tends to infinity with $\tau$
fixed, the parameter $\sigma$ in the averaging lemma can be chosen to tend to
infinity. Thus, by Lemma \ref{compactRhoSinfinity.lem}, the constants in the
averaging lemma tend to 1 as $s$ tend to infinity, from which the claimed
convergence of norms follows. The equalities of the various Hilbert spaces as
sets and the convergence of the norms allows us to deduce the unitarity of
$B_{\infty,\tau}$ from the unitarity of the maps $B_{s,\tau}$.
\end{proof}

\appendix

\section{Proof of Proposition \ref{nutIndep.prop}\label{technicalities.app}}

In this section, we provide a proof of Proposition \ref{nutIndep.prop}, which
we argued for heuristically in Section \ref{section s-->infinity}.

\begin{theorem}
\label{the.B.1.13} Let $G$ be a Lie group with Lie algebra $\mathfrak{g}$ and
fix an inner product on $\mathfrak{g}.$ For any subspace $V\subseteq
\mathfrak{g}$, define%
\[
\Delta_{V}=\sum_{j}\tilde{X}_{j}^{2},
\]
where $\{X_{j}\}$ is an orthonormal basis for $V,$ with domain $\mathcal{D}%
(\Delta_{V})=C_{c}^{\infty}(G)$. Then $\Delta_{V}$ is essentially self-adjoint
as an unbounded operator on $L^{2}(G,dg)$, where $dg$ is a right Haar measure.
Moreover, its closure $\bar{\Delta}_{V}$ is non-positive, and the associated
heat operators $e^{\frac{t}{2}\bar{\Delta}_{V}}$ are left-invariant for each
$t>0$.
\end{theorem}

We give here a proof based on work of J\o rgensen; a brief outline of a more
elementary argument was given in \cite[p. 950]{Driver2009a}, based on a method
communicated to the first author by L. Gross.

\begin{proof}
We fix a left Haar measure $m$ in addition to the right Haar measure $\lambda$
on $G$. Let $R$ be the unitary right regular representation on $L^{2}%
(G,\lambda)$, i.e. for $x\in G$ and $\varphi\in L^{2}(G,\lambda)$ let
\[
(R(x)\varphi)(y)=\varphi(yx)\quad\text{ for all }y\in G.
\]
For $f\in C_{c}^{\infty}(G)$ and $\varphi\in L^{2}(G,\lambda)$ we associate a
\textquotedblleft G\aa rding vector\textquotedblright, $g:=R(f)\varphi\in
L^{2}(G,\lambda)$, defined by%
\begin{align}
(R(f)\varphi)(y)  &  :=\int_{G}f(x)(R(x)\varphi)(y)\,dm(x)\nonumber\\
&  =\int_{G}f(x)\varphi(yx)\,dm(x)=\int_{G}f(y^{-1}x)\varphi(x)\,dm(x).
\label{e.rf}%
\end{align}
(According to a result of Malliavin and Dixmier \cite{Dixmier1978}, the space
of G\aa rding vectors coincides with the space of \textquotedblleft$C^{\infty
}$ vectors.\textquotedblright)

For $X\in\mathfrak{g}$ let $\hat{X}$ denote the right-invariant vector field
on $G$ which agrees with $X$ at the identity (as compared with the
left-invariant vector field $\tilde{X}$). By general theory in \cite[Theorem
1.1]{Jorgensen1975} or by direct computation, $R(f)\varphi\in C^{\infty
}(G)\cap L^{2}(G,\lambda)$ and
\begin{equation}
\widetilde{X}g=\widetilde{X}R(f)\varphi=R\left(  -\hat{X}f\right)  \varphi\in
C^{\infty}(G)\cap L^{2}(G,\lambda),\quad\forall\;X\in\mathfrak{g}.
\label{e.xg}%
\end{equation}

Let $\mathcal{D}(L_{1})$ denote the span of of the G\aa rding vectors and
$L_{1}:=L_{0}|_{\mathcal{D}(L_{1})}$. According to \cite[Theorem
1.1]{Jorgensen1975} with $U=R$, the operator $L_{1}$ is essentially
self-adjoint. To complete the proof it suffices to show $\bar{L}=\bar{L}_{1}$
and for this it suffices to show $L_{1}\subset\bar{L}$ and $L\subset\bar
{L}_{1}.$ We now verify the two desired operator inclusions.

\begin{itemize}
\item ($L_{1}\subset\bar{L}$) Let $g:=R(f)\varphi\in\mathcal{D}(L_{1})$ be a
G\aa rding vector as above. Choose a sequence $\{h_{n}\}_{n=1}^{\infty}\subset
C_{c}^{\infty}(G,[0,1])$ as in \cite[Lemma 3.6]{Driver1997} such that
$h_{n}=1$ on a Riemannian ball of radius $n$ relative to the left-invariant
Riemannian metric on $G$, and so $\sup_{x\in G}|Sh_{n}(x)|<\infty$ whenever
$S$ is any left-invariant differential operator on $G$. By the dominated
convergence theorem, the fact $Sg\in C^{\infty}(G)\cap L^{2}(G,\lambda)$ for
any left-invariant differential operator, $S$, on $G$ (see \eqref{e.xg}), and
the stated properties of $\{h_{n}\}_{n=1}^{\infty}$, it is easily shown that
$h_{n}g\rightarrow g$ and $L(h_{n}g)\rightarrow L_{1}g$ in $L^{2}(G,\lambda)$
as $n\rightarrow\infty$. This shows that $g\in\mathcal{D}(\bar{L})$ and
$\bar{L}g=L_{1}g$, i.e., $L_{1}\subset\bar{L}$.

\item ($L\subset\bar{L}_{1}$) Choose $\delta_{n}\in C_{c}^{\infty}%
(G,[0,\infty))$ such that $\int_{G}\delta_{n}(x)\,dm(x)=1$ for each $n$ and
$\mathrm{supp}(\delta_{n})\downarrow\{e\}$ as $n\rightarrow\infty$. Let
$\iota\colon G\rightarrow G$ denote the inverse map, i.e.\ $\iota(x)=x^{-1}$
for all $x\in G$. If $f\in C_{c}^{\infty}(G)$, then $g_{n}:=R(f\circ
\iota))\delta_{n}\rightarrow f$ in in $L^{2}(G,\lambda)$ as $n\rightarrow
\infty$ (see \eqref{e.rf}). Moreover, $g_{n}\in\mathcal{D}(L_{1}%
)\cap\mathcal{D}(L)$ and
\[
L_{1}g_{n}=R\left\{  \sum_{j=1}^{k}\hat{X}_{j}^{2}(f\circ\iota)\right\}
\delta_{n}=R\left(  (Lf)\circ\iota\right)  \delta_{n}\rightarrow
Lf,\quad\text{as }n\rightarrow\infty
\]
where the convergence is in $L^{2}(G,\lambda)$. Thus, it follows that
$f\in\mathcal{D}(\bar{L}_{1})$ and $\bar{L}_{1}f=Lf$, i.e., $L\subset\bar
{L}_{1}$.
\end{itemize}

This concludes the proof of self-adjointness. The non-positivity of the
self-adjoint extension $\bar{L}$ and the left invariance of the operators
$e^{t\bar{L}}$ are now standard exercises.
\end{proof}

\begin{lemma}
\label{lem.commy}Let $H$ be a separable Hilbert space, let $A$ and $B$ be two
essentially self-adjoint non-positive operators on $H$, and suppose $Q\colon
H\rightarrow H$ is a bounded operator such $QB\subseteq AQ$; i.e.,
$Q(\mathcal{D}(B))\subseteq\mathcal{D}(A)$ and $QB=AQ$ on $\mathcal{D}(B)$.
Then $Qe^{t\bar{B}}=e^{t\bar{A}}Q$ for all $t\geq0$.
\end{lemma}

\begin{proof}
If $f\in\mathcal{D}(\bar{B})$ and $f_{n}\in\mathcal{D}(B)$ such that
$f_{n}\rightarrow f$ and $Bf_{n}\rightarrow\bar{B}f$, then $Qf_{n}\rightarrow
Qf$ and $AQf_{n}=QBf_{n}\rightarrow Q\bar{B}f$ as $n\rightarrow\infty$.
Therefore it follows that $Qf\in\mathcal{D}(\bar{A})$ and $\bar{A}Qf=Q\bar
{B}f$ for all $f\in\mathcal{D}(\bar{B})$; i.e., $Q\bar{B}\subseteq\bar{A}Q$.
So for any $\lambda\in\mathbb{C}$ we may conclude that $(\lambda I-\bar
{A})Qf=Q(\lambda I-\bar{B})f$ for all $f\in\mathcal{D}(\bar{B})$. If we assume
$\lambda>0$ and $g\in H$, we may take $f=(\lambda I-\bar{B})^{-1}%
g\in\mathcal{D}(\bar{B})$ in the previous identity to find
\[
(\lambda I-\bar{A})Q(\lambda I-\bar{B})^{-1}g=Qg.
\]
Multiplying this equation by $(\lambda I-\bar{A})^{-1}$ and using the fact
that $g$ was arbitrary shows that $Q(\lambda I-\bar{B})^{-1}=(\lambda
I-\bar{A})^{-1}Q$ or, equivalently,%
\[
Q(I-\lambda^{-1}\bar{B})^{-1}=(I-\lambda^{-1}\bar{A})^{-1}Q\quad\text{ for all
}\;\lambda>0.
\]
A simple induction argument then shows that
\begin{equation}
Q(I-\lambda^{-1}\bar{B})^{-n}=(I-\lambda^{-1}\bar{A})^{-n}Q\quad\text{ for all
}\;\lambda>0. \label{equ.qcom}%
\end{equation}

Now, note that $\lim_{n\rightarrow\infty}(1-\frac{y}{n})^{-n}=e^{y}$ and
$0\leq(1-\frac{y}{n})^{-n}\leq1$ for $y\leq0$. We thus obtain the following
strong operator limits, using the spectral theorem and the dominated
convergence theorem:
\[
e^{t\bar{B}}=\lim_{n\rightarrow\infty}\left(  I-\frac{t}{n}\bar{B}\right)
^{-n}\quad\text{ and }\quad e^{t\bar{A}}=\mathrm{\!\!}\lim_{n\rightarrow
\infty}\left(  I-\frac{t}{n}\bar{A}\right)  ^{-n}.
\]
Therefore, taking $\lambda=n/t$ in \eqref{equ.qcom} and then letting
$n\rightarrow\infty$ shows $Qe^{t\bar{B}}=e^{t\bar{A}}Q$ for all $t>0$. This
completes the proof for $t>0$, and the $t=0$ case is immediate.
\end{proof}

\begin{corollary}
\label{c.B.1.14}If $K$ is a Lie subgroup of $G$, $V\subseteq\mathfrak{g}$ is
an $\mathrm{Ad}(K)$-invariant subspace, and $\left\langle \cdot,\cdot
\right\rangle _{V}$ is an $\mathrm{Ad}(K)$-invariant inner product on $V$,
then $e^{\frac{t}{2}\bar{\Delta}_{V}}$ commutes with right translations by
elements of $K$.
\end{corollary}

\begin{proof}
If $Q$ is a right-translation by an element of $K$ and $A=B=\Delta_{V}$ with
$\mathcal{D}(\Delta_{V})=C_{c}^{\infty}(G)$, then $QB=AQ$, and $Q$ preserves
$\mathcal{D}(\Delta_{V})$ in this case. The result now follow by an
application of Lemma \ref{lem.commy}.
\end{proof}

\begin{definition}
[$K$-averaging]\label{def.kavg}Let $P$ be the $K$\textbf{-averaging operator}
defined on $L_{\mathrm{loc}}^{1}(K_{\mathbb{C}})$ by
\[
\left(  Pf\right)  (z)=\int_{K}f(zk)\,dk
\]
where $dk$ denotes the Haar probability measure on $K$.
\end{definition}

Since the Haar measure on $K$ is invariant under inversion and the convolution
with itself is still Haar measure, we can easily check that $P\colon
L^{2}(K_{\mathbb{C}})\rightarrow L^{2}(K_{\mathbb{C}})$ is an orthogonal
projection. The operator $P$ also preserves the subspaces $C^{\infty
}(K_{\mathbb{C}})$ and $C_{c}^{\infty}(K_{\mathbb{C}})$ and if $f\in
C(K_{\mathbb{C}})$ we have $Pf(zk)=Pf(z)$ for all $k\in K$ and $z\in
K_{\mathbb{C}}$. Proposition \ref{nutIndep.prop} states, in this language,
that $P\mu_{s,\tau}=\nu_{t},$ where~$t=\operatorname{Re}\tau.$

\begin{proof}
[Proof of Proposition \ref{nutIndep.prop}]If $X\in\mathfrak{k}$ and $f\in
C_{c}^{\infty}(K_{\mathbb{C}})$, then $\left(  Pf\right)  (ze^{rX})=\left(
Pf\right)  (z)$ for all $z\in K_{\mathbb{C}}$ and $r\in\mathbb{R}$.
Differentiating at $r=0$ shows that $\widetilde{X}Pf=0$ for any $X\in
\mathfrak{k}$. Using the fact that $\tilde{X}_{j}\tilde{Y}_{j}=\tilde{Y}%
_{j}\tilde{X}_{j}$, which follows from the definition $Y_{j}=JX_{j}$ and
\eqref{Jbracket}, it follows from Definition \ref{d.Lap.rsu} that
\begin{equation}
\Delta_{s,\tau}P=\frac{t}{2}\Delta_{J\mathfrak{k}}P=P\frac{t}{2}%
\Delta_{J\mathfrak{k}}\;\text{ on }\;C_{c}^{\infty}(K_{\mathbb{C}}%
),\quad\text{ where}\;\Delta_{J\mathfrak{k}}:=\sum_{j=1}^{d}\tilde{Y}_{j}^{2}.
\label{equ.dst}%
\end{equation}
For the last equality, we have used that $\Delta_{J\mathfrak{k}}$ commutes
with right translations by elements of $K$ and therefore with $P.$ An
application of Lemma \ref{lem.commy} with $Q=P$, $A=\Delta_{s,\tau}$, and
$B=\frac{t}{2}\Delta_{J\mathfrak{k}}$ gives, $Pe^{\frac{t}{2}\bar{\Delta
}_{J\mathfrak{k}}}=e^{\bar{\Delta}_{s,\tau}}P$ for all $\tau\in\mathbb{D}%
\left(  s,s\right)  $ with $\operatorname{Re}\tau=t.$ In particular we may
conclude that
\begin{equation}
e^{\bar{\Delta}_{s,\tau}}P=e^{\bar{\Delta}_{s,t}}P\text{ }\quad\forall
~\tau=t+iu\in\mathbb{D}\left(  s,s\right)  \label{equ.comP}%
\end{equation}
or equivalently that%
\begin{equation}
\langle e^{\bar{\Delta}_{s,\tau}}Pv,w\rangle_{L^{2}(K_{\mathbb{C}})}=\langle
e^{\bar{\Delta}_{t,t}}Pv,w\rangle_{L^{2}(K_{\mathbb{C}})}\text{ }\quad
\forall~u,v\in C_{c}(K_{\mathbb{C}},\mathbb{R}). \label{e.convert.s.tau->t.t}%
\end{equation}

For the rest of the proof let $\bar{\mu}_{s,\tau}=P\mu_{s,\tau}$ be the
$K$-average of $\mu_{s,\tau}$. We may rewrite the left-hand-side of
\eqref{e.convert.s.tau->t.t} as,%
\begin{align*}
\langle e^{\bar{\Delta}_{s,\tau}}Pv,w\rangle_{L^{2}(K_{\mathbb{C}})}  &
=\int_{K_{\mathbb{C}}^{2}}\mu_{s,\tau}(g)(Pv)(zg)w(z)\,dg\,dz\\
&  =\int_{K_{\mathbb{C}}^{2}\times K}\mu_{s,\tau}(g)v(zgk)w(z)\,dg\,dz\,dk\\
&  =\int_{K_{\mathbb{C}}^{2}\times K}\mu_{s,\tau}(gk^{-1}%
)v(zg)w(z)\,dg\,dz\,dk\\
&  =\int_{K_{\mathbb{C}}^{2}\times K}\mu_{s,\tau}(gk)v(zg)w(z)\,dg\,dz\,dk\\
&  =\int_{K_{\mathbb{C}}^{2}}\bar{\mu}_{s,\tau}(g)v(zg)w(z)\,dg\,dz.
\end{align*}
This equation with $\tau=t$ also shows the right-hand-side of
\eqref{e.convert.s.tau->t.t} is given by
\[
\langle e^{\bar{\Delta}_{t,t}}Pv,w\rangle_{L^{2}(K_{\mathbb{C}})}%
=\int_{K_{\mathbb{C}}^{2}}\nu_{t}(g)v(zg)w(z)\,dg\,dz.
\]

Comparing the last two identities shows, for all $v,w\in C_{c}(K_{\mathbb{C}%
})$,
\[
\int_{K_{\mathbb{C}}^{2}}\bar{\mu}_{s,\tau}(g)v(zg)w(z)\,dg\,dz=\int%
_{K_{\mathbb{C}}^{2}}\nu_{t}(g)v(zg)w(z)\,dg\,dz.
\]
As $C_{c}(K_{\mathbb{C}})$ is dense in $L^{2}(K_{\mathbb{C}})$, we may
conclude that for all $v\in C_{c}(K_{\mathbb{C}})$,%
\[
\int_{K_{\mathbb{C}}}\bar{\mu}_{s,\tau}(g)v(zg)\,dg=\int_{K_{\mathbb{C}}}%
\nu_{t}(g)v(zg)\,dg\quad\text{for a.e. }z
\]
and hence for every $z\in K_{\mathbb{C}}$ as both sides of the previous
equation are continuous in $z$. Thus, taking $z=e$, it follows that,
\[
\int_{K_{\mathbb{C}}}\bar{\mu}_{s,\tau}(g)v(g)\,dg=\int_{K_{\mathbb{C}}}%
\nu_{t}(g)v(g)\,dg\text{ }\forall~v\in C_{c}(K_{\mathbb{C}},\mathbb{R}).
\]
So as above, the density of $C_{c}(K_{\mathbb{C}})$ in $L^{2}(K_{\mathbb{C}})$
along with the continuity of both $\bar{\mu}_{s,\tau}$ and $\nu_{t},$ allows
us to conclude that $\bar{\mu}_{s,\tau}(g)=\nu_{t}(g)$ for all $g\in
K_{\mathbb{C}}$.
\end{proof}

\subsection*{Acknowledgments}

This project began as the result of a conversation between the third author
and Thierry L\'{e}vy at Oberwolfach in June, 2015, regarding the idea of
classifying all $\mathrm{Ad}(\mathrm{U}(n))$-invariant inner products on
$\mathrm{GL}(n)$ (and studying their large-$n$ limits). This led the third
author to prove Theorem \ref{t.Ad(K)-inv.inn.prod}, and consequently to wonder
if this extension of the two-parameter family of inner products studied in
\cite{Kemp2015b} was associated to some kind of \textquotedblleft twisted
Segal--Bargmann transform\textquotedblright\ extending the one in
\cite{Driver1999} and \cite{Hall2001b}.

The authors thank the referee for detailed comments that helped streamline the presentation.

\bibliographystyle{acm}
\bibliography{C-Time-SBT}

\end{document}